\documentclass[12pt]{amsart}

\usepackage{amsmath,amsfonts,amsthm,amsopn,cite,mathrsfs, amssymb}
\usepackage{epsfig,verbatim,color}
\usepackage{subcaption}


\newcommand{\Rdst}{{\mathbb{R}^d}}
\newcommand{\Rst}{{\mathbb{R}}}

\newcommand{\Nst}{{\mathbb{N}}}

\newcommand{\Zdst}{{\mathbb{Z}^d}}

\newcommand{\norm}[1]{\lVert#1\rVert}
\newcommand{\bC}{{\mathbb{C}}}

\def\supp{\operatorname{supp}}

\newcommand{\sep}{\mathop{\mathrm{sep}}}
\newcommand{\rel}{\mathop{\mathrm{rel}}}
\newcommand{\wmes}{{W(\mathcal{M}, L^\infty)}}

\newcommand{\map}{\tau}

\newcommand{\mapj}[1]{\map_j(#1)}

\newcommand{\weakconv}{\xrightarrow{w}}
\newcommand{\lipconv}{\xrightarrow{Lip}}

\newcommand{\Cn}{\mathbb{C}^n}

\newcommand{\func}{F}
\newcommand{\Func}{\mathrm{F}}

\newcommand{\env}{\Theta}

\newcommand{\ip}[2]{\ensuremath{\left<#1,#2\right>}}
\newcommand{\sett}[1]{\ensuremath{\left \{ #1 \right \}}}
\newcommand{\abs}[1]{\ensuremath{\left| #1 \right| }}

\newtheorem{theorem}{Theorem}[section]
\newtheorem{lemma}[theorem]{Lemma}
\newtheorem{prop}[theorem]{Proposition}
\newtheorem{cor}[theorem]{Corollary}
\newtheorem{definition}[theorem]{Definition}
\newtheorem{rem}[theorem]{Remark}

\usepackage[
textwidth=16.0cm,
textheight=22cm,
hmarginratio=1:1,
vmarginratio=1:1]{geometry}
\begin{document}

\title [Strict density inequalities]{Strict density inequalities for
  sampling and interpolation in weighted spaces of holomorphic
  functions}
\begin{abstract}
Answering a question of Lindholm, we prove strict density inequalities
for sampling and interpolation in Fock spaces of entire functions
in several complex variables defined by a plurisubharmonic weight. In
particular, these spaces do not admit a set that is simultaneously
sampling and interpolating.  To prove optimality of the density conditions, we construct
sampling sets with a density arbitrarily close to the critical  density.

The techniques combine methods from several complex variables
(estimates for $\bar \partial$) and the theory of localized frames in general
reproducing kernel Hilbert spaces (with no analyticity assumed). The
abstract results on Fekete points and deformation of frames may be  of independent interest.
\end{abstract}
\author[K. Gr\"{o}chenig] {Karlheinz Gr\"{o}chenig}
\address{Faculty of Mathematics \\
University of Vienna \\
Oskar-Morgenstern-Platz 1 \\
A-1090 Vienna, Austria}
\email{karlheinz.groechenig@univie.ac.at}

\author[A. Haimi]{Antti Haimi}
\address{Acoustics Research Institute, Austrian Academy of Sciences,
Wohllebengasse 12-14 A-1040, Vienna, Austria}
\email{ahaimi@kfs.oeaw.ac.at}

\author[J. Ortega-Cerd\`{a}]{Joaquim Ortega-Cerd\`{a}}
\address{Dept.\ Matem\`atiques i Inform\`atica,
Universitat  de Barcelona, Gran Via 585, 08007 Bar\-ce\-lo\-na, Spain}
\email{jortega@ub.edu}

\author[J. L. Romero]{Jos\'{e} Luis Romero}

\address{Faculty of Mathematics,
University of Vienna,
Oskar-Morgenstern-Platz 1,
A-1090 Vienna, Austria\\and\\
Acoustics Research Institute, Austrian Academy of Sciences,
Wohllebengasse 12-14 A-1040, Vienna, Austria}
\email{jose.luis.romero@univie.ac.at,
jlromero@kfs.oeaw.ac.at}

\thanks{A. H. and J. L. R. gratefully acknowledge support from the WWTF grant INSIGHT (MA16-053).
J. L. R. gratefully acknowledges support from the Austrian Science Fund (FWF): P 29462 - N35.
A. H. was supported by the Austrian Science Fund (FWF) grant P 31153-N35.
J. O-C was supported by projects MTM2017-83499-P from the Ministerio de
Econom\'{\i}a y Competitividad, Gobierno de Espa\~na and by the Generalitat de
Catalunya (project 2017 SGR 358).
}

\maketitle

\section{Introduction and results}
Let $\phi: \Cn \to \mathbb{R} $ be a plurisubharmonic function, and assume that there are
constants $m, M > 0$ such that
\begin{equation} \label{eq:subharm_bounds}
i m \partial \bar \partial |z|^2 \leq i \partial \bar \partial \phi \leq M i \partial \bar \partial |z|^2
\end{equation}
in the sense of positive currents~\cite{GL86}.
For $1 \leq p < \infty$, we let $A^p_\phi$ be the space of entire functions on $\Cn$ equipped with
the norm
$$
\Vert f \Vert_{\phi,p}^p := \int_{\Cn} |f(z)|^p e^{-p \phi(z)} dm(z),
$$
where $dm$ denotes the Lebesgue  measure. For $p = \infty$, we use the norm
$$
\Vert f \Vert_{\phi, \infty} = \sup_{z \in \Cn} |f(z)| e^{-\phi(z)}.
$$
Point evaluations are bounded linear  functionals, and therefore
$A^2_\phi$ is a reproducing kernel Hilbert space. We denote its reproducing kernel by
$K_{\phi}(z,w)$, or just $K(z,w)$ when it is not ambiguous. We also write
$K_{\phi,w}(z) := K_{\phi}(z,w)$.

A set $\Lambda \subseteq \Cn$ is called a \emph{sampling set} for
$A^p_{\phi}$ if there are constants $A,B>0$ such that
\begin{align*}
A \Vert f \Vert^p_{\phi,p} \leq \sum_{\lambda \in \Lambda} \abs{f(\lambda)}^p e^{-p\phi(\lambda)}
\leq B \Vert f \Vert^p_{\phi,p}, \qquad \text{ for all } \, f \in A^p_\phi,
\end{align*}
with the usual modification for $p=\infty$. The  constants $A,B$  are
called the  stability
constants.
A set $\Lambda$ is called an \emph{interpolating set} for $A^p_\phi $,
if for every
$a \in \ell^p_{\phi}(\Lambda )$ there exists a function $f \in A^p_{\phi}$ such that
\begin{align*}
f(\lambda) = a_\lambda, \qquad \lambda \in \Lambda.
\end{align*}
In this case, there is always $C_p>0$
and a choice of $f$ such that $\norm{f}_{\phi,p} \leq C_p \norm{a}_{p,\phi}$.

This article is concerned with the density of sampling and interpolating sets.
The upper and lower \emph{weighted Beurling upper densities} of $\Lambda$ are defined by
\begin{align}
\label{eq_bd}
D^+_\phi(\Lambda): &=  \limsup_{r \to \infty} \sup_{z \in \Cn} \frac{ \# ( \Lambda \cap B_r(z))} {
\int_{B_r(z)} K(w,w) e^{-2\phi(w)} dm(w)},
\\
\label{eq_bd2}
D^-_\phi(\Lambda): &=  \liminf_{r \to \infty} \inf_{z \in \Cn} \frac{ \# ( \Lambda \cap B_r(z))} {
\int_{B_r(z)} K(w,w) e^{-2\phi(w)} dm(w)}.
\end{align}
For the standard weight $\phi(z)=\pi \abs{z}^2/2$, one recovers
Beurling's classical densities, since the reproducing kernel is
$K(w,z) = e^{\pi \bar{w} z}$.
In dimension $n=1$, sampling and interpolating sets are characterized
completely by
density conditions \cite{seip92, ortega1998beurling, marco2003interpolating}. In higher dimensions, only the necessity
of the conditions can be expected to hold. We investigate this matter
in several directions.

For $2$-homogeneous weights, Lindholm \cite{lindholm2001sampling}
showed the following necessary density conditions:
\begin{align}
  \label{eq:aaaa}
& \mbox{ If $\Lambda$ is a sampling set for } A^p_\phi, \mbox{ then }
  D^-_\phi(\Lambda) \geq 1 \, .
\\
\label{eq:bbbb}
&\mbox{ If $\Lambda$ is an interpolating set for } A^p_\phi, \mbox{
                  then } D^{+}_\phi (\Lambda) \leq 1 \, .
\end{align}
We will show that the homogeneity of $\phi $ may be removed and that
the necessary conditions \eqref{eq:aaaa} and \eqref{eq:bbbb}
are  valid for all  weights $\phi$ satisfying
\eqref{eq:subharm_bounds}.  This more general result follows from
 the abstract density theory \cite{bacahela06, fuhr2016density}, which
 is applicable due to  the off-diagonal decay of the reproducing
 kernel. See Section \ref{sec_nonstr} for
 the details.

Strictly speaking, the densities in \eqref{eq_bd} and \eqref{eq_bd2}
differ from those used in
\cite{ortega1998beurling, lindholm2001sampling}, where  the following
densities are used instead:
\begin{align*}
\widetilde D^+_\phi(\Lambda)&: =  \limsup_{r \to \infty} \sup_{z \in \Cn} \frac{
\# ( \Lambda \cap B_r(z))} {
\int_{B_r(z)} (i \partial \bar \partial \phi)^n},
\\
\widetilde D^-_\phi(\Lambda)&: =  \liminf_{r \to \infty} \inf_{z \in \Cn} \frac{
\# ( \Lambda \cap B_r(z))} {
\int_{B_r(z)} (i \partial \bar \partial \phi)^n}.
\end{align*}
It can be shown by combining
results of \cite{ortega1998beurling}, \cite{berndtsson1995interpolation} and
\cite{fuhr2016density} that in $1$ dimension we have
\begin{equation} \label{eq:densityconnection}
\widetilde D^+_\phi(\Lambda) =\frac{1}{\pi^n n!}D^+_\phi(\Lambda),
\end{equation}
and similarly for the lower densities. The relation \eqref{eq:densityconnection} holds also in several variables
if we assume in addition that $\phi$ is $2$-homogeneous
\cite{lindholm2001sampling} - see also \cite[Section 5.4]{fuhr2016density}.
In general, it remains an open problem to decide when both densities coincide.

As our first contribution we  show that these results are sharp.
\begin{theorem}\label{sharp}
Let $\phi: \Cn \to \mathbb{R} $ be a plurisubharmonic function satisfying
\eqref{eq:subharm_bounds}. Then given $\varepsilon>0$, there exists a set $\Lambda \subseteq \Cn$ that is
interpolating for $A^2_{(1+\varepsilon)\phi}$ and sampling for $A^2_{(1-\varepsilon)\phi}$. As a consequence,
\begin{align}
\label{eq_infsup}
\inf_{\Lambda \in SS} D_\phi^{+}(\Lambda)=
\sup_{\Lambda \in SI} D_\phi^{-}(\Lambda)=1 \, ,
\end{align}
where the infimum runs over all sampling sets for $A_\phi ^2$ and the
supremum over all interpolation sets for $A_\phi ^2$.
\end{theorem}

The sampling part of Theorem \ref{sharp} is closely related to
the main result in \cite{bacala11}, which establishes the existence of
frames in an abstract setting whose density is arbitrarily close to
the critical density. Although it may
be possible to apply the results of \cite{bacala11} to our setting,
 it is far from clear how to overcome certain technical challenges
caused by  the subtleties of general plurisubharmonic weights, such as
the construction of an adequate ``reference frame'' or the
identification of the corresponding  abstract densities with
$D^{\pm}_\phi$.
Instead, in this paper we resort to a new technique based on Fekete points introduced
 in \cite{Nir}. This approach  also yields the existence of
 interpolating sets with density arbitrary  close to the  critical
  density.

Our second contribution  is to show that the supremum and infimum in
\eqref{eq_infsup} are not attained.  We will  prove that the inequalities in \eqref{eq:aaaa} and
\eqref{eq:bbbb} are in fact strict. This question was mentioned as  an
open problem  by Lindholm~\cite{lindholm2001sampling}, but remained
unanswered  even for the special case of  $2$-homogeneous weights.

\begin{theorem}
\label{th_intro}
Let $\phi: \Cn \to \mathbb{R} $ be a plurisubharmonic function satisfying
\eqref{eq:subharm_bounds} and let  $p \in [1,\infty]$.
\begin{itemize}
\item[(a)] If $\Lambda$ is a sampling set for $A^p_\phi$, then
$D^{-}_\phi(\Lambda)>1$.
\item[(b)] If $\Lambda$ is an interpolating set for $A^p_\phi$,
then $D^{+}_\phi (\Lambda)<1$.
\end{itemize}
\end{theorem}

\begin{cor}
  There does not exist a set $\Lambda \subseteq \Cn$ that is
  simultaneously sampling and interpolating for $A_\phi
  ^2$. Equivalently, there is no Riesz basis for $A_\phi ^2$ that consists of
  reproducing kernels.
\end{cor}

The strict density conditions of Theorem~\ref{th_intro} are
substantially different from the necessary, non-strict conditions
in~\eqref{eq:aaaa} and ~\eqref{eq:bbbb}. So far, strict density
conditions have been proved only in few situations, namely  (i) for
weighted Fock spaces $A^p_\phi $ \emph{in dimension $1$}
\cite{ortega1998beurling}, and (ii) for
Gabor frames, where the result is known as the Balian-Low
theorem. Although there is an extensive literature on the theorem for
Gabor frames over a lattice, strict density conditions for non-uniform
Gabor frames were shown only recently in~\cite{AFK14} (with
pseudodifferential operators) and \cite{grorro15}.

We note that the proofs in  \cite{ortega1998beurling} for weighted Fock spaces in one variable
rely on the sufficiency of density conditions for sampling and
interpolation \cite{berndtsson1995interpolation}, and these are not available
in higher dimension. We will adopt a different  approach that  combines the  strategies of
\cite{ortega1998beurling} and  of \cite{grorro15}.
To circumvent arguments that are specific to one-dimensional complex
analysis, we resort instead to techniques from \cite{grorro15} that
were introduced originally  to study the stability of Gabor frames
under quite general deformations. More
precisely, we consider the notion of \emph{Lipschitz convergence of
  sets}. Roughly, a sequence of sets $\Lambda_j\subseteq \Cn$ converges
Lipschitz-wise to $\Lambda \subseteq \Cn$, if
there exist maps $\tau _j: \Lambda \to \Cn$, such that $\Lambda _j =
\tau _j (\Lambda)$, $\tau_j(\lambda ) \to \lambda $ for all $\lambda
\in \Lambda $ and distances are preserved locally. See Section
\ref{sec_lip} for the precise, technical definition.
Our third contribution is the following deformation result for
sampling sets (interpolating sets) in weighted Fock spaces.

\begin{theorem}
\label{th_lip_main}
Let $\phi: \Cn \to \mathbb{R} $ be a plurisubharmonic function satisfying
\eqref{eq:subharm_bounds}, and $p \in [1,\infty]$. Assume that
$\Lambda _j$ is a sequence of sets that converge to $\Lambda
$ in Lipschitz-wise,  $\Lambda_j \lipconv \Lambda$.

(a) If $\Lambda$ is a sampling set for $A^p_\phi$,
 then  $\Lambda_j$ is also a sampling
set for $A^p_\phi$ for sufficiently large $j$.

(b) If $\Lambda$ is an interpolating set for $A^p_\phi$,  then  $\Lambda_j$ is also an
interpolating set for $A^p_\phi$ for sufficiently large $j$.
\end{theorem}

 Theorem~\ref{th_intro} then  follows from the
deformation stability by choosing a sequence of dilated sets $\Lambda
_j = (1+\tfrac{1}{j})\Lambda $. If $\Lambda $ is a sampling set for
$A^p_\phi $, then so is $\Lambda _j$ for large $j$. Then $D^-_\phi
(\Lambda ) >D^- _\phi (\Lambda _j )$ and by the
necessary density condition \eqref{eq:aaaa} we obtain  $D^-_\phi
(\Lambda ) >D^- _\phi (\Lambda _j )\geq 1$. See Section \ref{sec_th_intro}
for details.

For the proof of the main results we will enrich the outline of
\cite{ortega1998beurling} and \cite{grorro15} by several new aspects:

(i) \emph{Universality.} $\Lambda $ is a sampling set (interpolating set) for
$A^p_\phi $ for \emph{some} $p\in [1,\infty ]$, if and only if $\Lambda $ is a sampling set (interpolating set) for
$A^p_\phi $ for \emph{all} $p\in [1,\infty ]$. The technical novelty
is a Wiener-type lemma for infinite  matrices with off-diagonal decay
that are left-invertible on a subspace (Lemma~\ref{th_wiener_sub}). This clarifies some subtleties
in~\cite{albakr08} and \cite{grorro15}, and enhances their applicability (in \cite{grorro15} we used
so-called Wilson bases to reach similar conclusions).

(ii) \emph{Weak limits} play an important role in sampling theory and
in deformation results. The main obstacle in weighted Fock spaces
 is their lack of translation invariance. This was circumvented in \cite{ortega1998beurling}
by noting that, while a single  space $A^p_\phi$ may not be translation invariant,
the union of all $A^p_\phi $ is translation invariant. This insight
was leveraged by developing abstract translation operators
that map a weighted Fock space into another weighted Fock space. The
extension of these ideas to several complex variables requires
considerable technicalities (Section~4). In particular, we  will show
that the map $\phi \mapsto K_\phi $ (for the reproducing kernel of
$A^2_\phi $) obeys some continuity property (Proposition~\ref{prop:trans_lim}).

(iii) The \emph{theory of localized frames in reproducing kernel Hilbert
  spaces} enters several times in the proof of the universality of
sampling sets and of Theorem~\ref{th_lip_main}. These  arguments  do not rely on
analyticity and may be of independent interest for further
applicability. In particular, Theorem~\ref{th_fekete} contains an abstract
version of the construction of Fekete points in reproducing kernel
Hilbert spaces.

Finally we comment on a question raised in  \cite{ortega1998beurling}
on the  difference between Paley-Wiener and Fock spaces. Whereas the
necessary density condition in weighted Fock spaces is strict, it is
not so in the Paley-Wiener space, and consequently, Paley-Wiener space
admits sequences that are both sampling and interpolating. We believe
that the difference lies  in the  off-diagonal decay of reproducing
kernels.
Whereas the reproducing kernel of
Paley-Wiener space is not even in $L^1$, the adjusted reproducing kernel of
weighted Fock space decays exponentially. In the end, this difference
may contribute to the different behavior of the two spaces.

The article is organized as follows. In Section \ref{sec_otation} we collect some facts about Fock spaces and
$\bar \partial$-equations, while in Section \ref{sec_tools} we introduce the key definitions and tools. The abstract
translation operators are introduced in Section \ref{sec_ops}. These tools are used to characterize sampling
and interpolating sets in Section \ref{sec_char}. Theorem \ref{th_intro} is derived in Section \ref{th_stab}.
For clarity, the more general arguments that are applicable to
abstract reproducing kernel Hilbert spaces  with a certain
off-diagonal decay are postponed to Section \ref{sec_loc}.

\section{Preliminaries}
\label{sec_otation}
\subsection{Notation}
We are mainly interested in functions of $n$ complex variables, but we develop some
auxiliary results on the Euclidean spaces $\Rdst$. Of course, when we apply these to $\Cn$, we
let $d=2n$.

A set $\Lambda \subseteq \Rdst$ is called \emph{relatively separated} if
\begin{align}
\label{eq_rel}
\rel(\Lambda) := \sup \{ \#(\Lambda \cap B_1(x)) : x \in \Rdst \} < \infty,
\end{align}
and it is called \emph{separated} if
\begin{align}
\label{eq_sep}
\sep(\Lambda) := \inf \sett{\abs{\lambda-\lambda'}: \lambda \not = \lambda' \in \Lambda} > 0.
\end{align}
Separated sets are relatively separated, and relatively separated sets are finite unions of
separated sets.
A set $\Lambda$ is called \emph{relatively dense} if there exists $R>0$ such that
$\Rdst = \bigcup_{\lambda \in \Lambda} B_R(\lambda)$.

\subsection{The reproducing kernel}
\label{sec_off}
Recall that we denote the reproducing kernel of $A_\phi^2$ by $K_{\phi}$ and write
$K_{\phi,
z}(w)= K_{\phi}(w,z)$. The diagonal of the reproducing kernel
satisfies
\begin{align}
\label{eq_diag}
0<c \leq K_\phi(z,z)e^{-2\phi(z)} \leq C <\infty,
\end{align}
for some constants $c,C$ that only depend on the constants in \eqref{eq:subharm_bounds};
see e.g. \cite[Proposition 2.5]{schva12}. In addition, the reproducing
kernel $K_\phi $ satisfies the following
off-diagonal decay estimate  \cite{delin1998pointwise}:
\begin{equation} \label{eq:offdiag-bound}
|K_\phi(z, w)|e^{-\phi(z)-\phi(w)} \leq C e^{-c|z-w|},
\end{equation}
for all $z, w \in \Cn$ and some constants $c, C > 0$ which only depend on the bounds in
\eqref{eq:subharm_bounds}. See \cite{dall2015pointwise} for more
general conditions
for off-diagonal decay.

\subsection{Non-strict density conditions}
\label{sec_nonstr}
The following statement offers a small extension of Lindholm's density
theorem~\cite{lindholm2001sampling}.
\begin{theorem}
\label{th_non_strict}
Let $\phi: \Cn \to \mathbb{R} $ be a plurisubharmonic function satisfying
\eqref{eq:subharm_bounds}.
\begin{itemize}
\item[(a)] If $\Lambda$ is a sampling set for $A^2_\phi$, then
$D_\phi^{-}(\Lambda) \geq 1$.
\item[(b)] If $\Lambda$ is an interpolating set for $A^2_\phi$,
then $D_\phi^{+}(\Lambda) \leq 1$.
\end{itemize}
\end{theorem}
\begin{proof}
We apply the abstract density result in \cite[Corollary 4.1]{fuhr2016density}
to the metric space $\Cn=\Rst^{2d}$, $d\mu(z) = dm(z)$, and the reproducing kernel Hilbert
space
\begin{align*}
V^2 := \sett{ e^{-\phi} f: f \in A^2_\phi }.
\end{align*}
The density theorem  \cite[Corollary 4.1]{fuhr2016density} requires  certain assumptions on
the metric and measure, which are indeed satisfied by the Euclidean
space and Lebesgue measure, and assumptions on
the reproducing kernel
of $V^2$ (behavior of the diagonal and off-diagonal decay), which is
$K_\phi(z, w)e^{-\phi(z)-\phi(w)}$. The required conditions on the
reproducing kernel are easily seen to hold
due to \eqref{eq_diag} and \eqref{eq:offdiag-bound}.
\end{proof}
\begin{rem} \rm
For $2$-homogeneous weights $\phi$, Theorem \ref{th_non_strict} is essentially
due to Lindholm \cite{lindholm2001sampling}. To be precise, the
density condition in \eqref{eq_bd} and \eqref{eq_bd2} and those of Lindholm are
formally different, but they were shown  to coincide for
$2$-homogeneous weights by \cite[Section 5.4]{fuhr2016density}. The
generalization in Theorem~\ref{th_non_strict} may be taken as a hint
that the new notion of density  is perfectly  appropriate for weighted
Fock spaces.
\end{rem}

By \eqref{eq_diag}, the unweighted Beurling density $D^{\pm}(\Lambda )$ is
comparable to the weighted density $D_\phi ^{\pm}(\Lambda )$, namely
$cD^{\pm} (\Lambda ) \leq D^{\pm}_\phi (\Lambda ) \leq C D^{\pm}
(\Lambda )$. Since a set of positive lower Beurling density is
relatively dense, we obtain the following corollary.
\begin{cor}
\label{cor_rd}
Assume that $\phi$ satisfies \eqref{eq:subharm_bounds}. Then
every sampling set for $A^2_\phi$ is relatively dense.
\end{cor}

\subsection{The $\partial \bar \partial$ equation}
\begin{lemma} \label{lem:poincare-lelong}
Let $\theta = \sum_{1 \leq j,k \leq n} \theta_{jk} dz_j \wedge d\bar z_k$ be a positive,
$d$-closed $(1,1)$-current
satisfying $\theta \leq Mi \partial \bar \partial |z|^2$.
Then there exists $u:\Cn \to \bC$
solving the equation $i \partial \bar \partial u = \theta$,
and such that
\begin{equation}
|u(z)| \leq C M (1+|z|)^2 \log (1+ |z|),
\end{equation}
where the constant $C$ depends only on the dimension $n$.
\end{lemma}
\begin{proof}
The solution is found in two stages. In the first step, we let $d= \partial + \bar \partial$
and solve $dv=
\theta$. The solution $v$ is as in Poincar\'e's lemma: $v= v_{0,1}+ v_{1,0}$ where
$$
v_{0,1}(z)= \sum_{1 \leq j,k \leq n} \bigg( \int_{t=0}^1 \theta_{j,k}(tz) t z_j dt \bigg) d
\bar{z}_k, \quad v_{1,0}= -\overline{v_{0,1}}.
$$
We have $\bar \partial v_{0,1}=\partial v_{1,0}=0$ and  check easily  that $dv= \theta$.
Furthermore, $v$ satisfies the estimate $v(z) \leq CM(1+ |z|)$, where $C$ depends
 only on the dimension of the space.

In the second step, we solve the equation $i \bar \partial w = v_{0,1}$. By
\cite[Theorem 9']{berndtsson1982henkin}, there exists a solution to this equation, given by an
explicit integral formula, satisfying
$$
| w(z)| \leq C M (1+|z|)^2 \log (1+ |z|).
$$
Now, it is readily checked $u:= 2 \mathrm{Re}\,  w $ solves the equation $i \partial \bar \partial
u
= \theta$
and also satisfies the desired growth estimate.
\end{proof}
\begin{rem} \rm
It follows from standard regularity theory for the Poisson equation that the solution $u$ in Lemma \ref{lem:poincare-lelong}
has derivatives of order 1 that are locally $\alpha$-H\"older
continuous, i.e.,  $u\in
C^{1,\alpha }$  for every  $\alpha \in (0,1)$.
\end{rem}
\subsection{Size control}
\label{sec_assumptions}
In some of the results we introduce the following extra size assumption on $\phi$:
\begin{align}\label{eq_a1}
|\phi(z)| \leq C (1+|z|)^2 \log (1+ |z|), \qquad z \in \Cn.
\end{align}
As explained below, this extra condition can always be achieved without changing the class of sampling or interpolating sets.
\begin{prop}
\label{prop_a1}
Let $\phi: \Cn \to \mathbb{R} $ be a plurisubharmonic function satisfying
\eqref{eq:subharm_bounds}. Then there exists
a plurisubharmonic function $\tilde\phi: \Cn \to \mathbb{R}$
satisfying \eqref{eq:subharm_bounds} and \eqref{eq_a1},
and with the following property: a set $\Lambda \subseteq \Cn$ is a sampling (resp. interpolating) set for
$A^p_\phi$ and some $p \in [1,\infty]$
if and only if $\Lambda$ is a sampling (resp. interpolating) set for
$A^p_{\tilde\phi}$. Furthermore, the density remains invariant under
this change, $D_\phi ^\pm (\Lambda ) = D_{\tilde{\phi}} ^\pm (\Lambda
)$.
\end{prop}
\begin{proof}
Lemma \ref{lem:poincare-lelong} provides a function $\tilde\phi$ satisfying \eqref{eq_a1},
and such that $\partial \bar \partial \tilde \phi = \partial \bar \partial \phi $. This
implies that there exists an entire function $G$ such that
$\mathrm{Re}\, G = \tilde\phi-
\phi$, and therefore,
$$
e^{-\phi(z)} = |e^{G(z)}| e^{-\tilde\phi(z)}.
$$
Hence, multiplication by $e^G$ gives an isometry from $A^p_\phi$ to $A^p_{\tilde\phi}$. The
sampling and interpolating sets are therefore the same in $A^p_\phi$ and
$A^p_{\tilde\phi}$. The kernels of $A^2_\phi$ and $A^2_{\tilde{\phi}}$
are related by
$$
K_{\tilde{\phi}}(z,w) = e^{G(z)-\overline{G(w)}} \,  K_\phi (z,w) \,.
$$
Consequently the Bergman measure of the ball $B_r(z)$ in the
definition of the density is
\begin{align*}
\int _{B_r(z)} K_{\tilde{\phi}}(w,w) e^{-2\tilde{\phi}(w)} \, dm(w) &=
\int _{B_r(z)} K_{\phi}(w,w) \, |e^{G(w)}|^2  e^{-2\tilde{\phi}(w)} \,
                                                                      dm(w)\\
  & = \int _{B_r(z)} K_{\phi }(w,w) e^{-2\phi (w)} \, dm(w) \, ,
\end{align*}
and thus $D^{\pm} _{\tilde{\phi}}(\Lambda ) = D^{\pm}_\phi (\Lambda
)$.
\end{proof}
\section{Some tools}
\label{sec_tools}
\subsection{Bessel bounds for weighted analytic functions}
The following lemma follows from \cite[Lemmas 7 and 17]{lindholm2001sampling}.

\begin{lemma} \label{lem:point_eval}
Let $f$ be a holomorphic function on $B_1(z) \subseteq \Cn$.
Let $\psi: B_1(z) \to \Rst$ be a plurisubharmonic function such that
$i \partial \bar \partial
\psi(w) \leq M i \partial \bar \partial |w|^2$.
Then, for all $p \in [1,\infty)$,
\begin{equation}
  \label{eq:c1}
|f(z)|^p e^{-p\psi(z)} \leq {C_1} \int_{B_1(z)} |f(w)|^p e^{-p\psi(w)}
dm(w).
\end{equation}
In addition, if $f(z) \neq 0$, then for $r > 0$
\begin{equation}
  \label{eq:c2}
|\nabla \left(|f(z)|^r e^{-r\psi(z)} \right)| \leq {C_2} \bigg[ \int_{B_1(z)}
|f(w)|^p e^{-p\psi(w)} dm(w) \bigg]^{r/p}.
\end{equation}
The constants depend only on $p, r, M$, and the dimension $n$.
\end{lemma}

As a consequence of \eqref{eq:c1}, we obtain the following local Bessel bound.
\begin{cor}
\label{coro_besse}
Let $\Lambda \subseteq \Cn$ be relatively separated and let $f$ be a holomorphic
function on
$\Lambda+
B_1(0)$,  and suppose that $\psi$ is a plurisubharmonic function satisfying $i
\partial
\bar \partial \psi(w) \leq M i
\partial \bar \partial |w|^2$ on $\Lambda+B_1(0)$.
Then, for all $p \in [1,\infty)$,
\begin{align*}
\left( \sum_{\lambda \in \Lambda} \abs{f(\lambda)}^p e^{-p\psi(\lambda)}
\right)^{1/p} \leq C \rel(\Lambda)
\bigg[ \int_{\Lambda+B_1(0)} |f(w)|^p e^{-p\psi(w)} dm(w) \bigg]^{1/p}.
\end{align*}
\end{cor}
We also derive the following fact.
\begin{cor}
\label{cor_int_set}
Let $\psi$ be a plurisubharmonic function such that $i \partial \bar \partial
\psi(w) \leq M i
\partial \bar \partial |w|^2$, and let  $1 \leq p \leq +\infty$. Then
every interpolating set for $A^p_\psi$ is separated.
\end{cor}
\begin{proof}
As in \cite[Proposition 9]{ortega1998beurling},
if $\Lambda$ is interpolating and $\lambda, \lambda' \in \Lambda$ are two
different points, we may find $f \in A^p_\psi$ such that
$f(\lambda)=e^{\psi(\lambda)}$, $f(\lambda')=0$, and
$\norm{f}_{p,\psi} \leq C_\Lambda$.  Lemma
\ref{lem:point_eval} now implies that
$1=\Big|\abs{f(\lambda)e^{-\psi(\lambda)}}-\abs{f(\lambda')e^{-\psi(\lambda')}}\Big|
\lesssim C_\Lambda \abs{\lambda-\lambda'}$.
\end{proof}
\subsection{Amalgam spaces}
\label{sec_am}
The \emph{amalgam space} $W(L^\infty,L^1)(\Rdst)$ consists of all functions $f
\in
L^\infty(\Rdst)$
such that
\begin{align*}
\norm{f}_{W(L^\infty,L^1)} := \int_{\Rdst} \norm{f}_{L^\infty(B_1(x))} dx
\asymp
\sum_{k \in \Zdst} \norm{f}_{L^\infty([0,1]^d+k)} <\infty.
\end{align*}
The (closed) subspace of $W(L^\infty,L^1)(\Rdst)$ of continuous functions is
denoted
$W(C_0,L^1)(\Rdst)$, and is a convenient space  of test
functions. Its dual space will be denoted $\wmes(\Rdst)$ and consists
of all complex-valued Borel
measures $\mu$ such that
\begin{align*}
\norm{\mu}_\wmes := \sup_{x \in \Rdst} \norm{\mu}_{B_1(x)}
= \sup_{x \in \Rdst} \abs{\mu}(B_1(x)) < \infty.
\end{align*}
Such measures are often called \emph{translation-bounded}.
We refer the reader to \cite{feichtinger1980banach}
for a general theory of Wiener amalgam spaces.

\subsection{Universality of sampling and interpolating sets}
The following universality results are a central technical tool.

\begin{theorem}
\label{thm_wiener_samp}
Assume that $\phi$ satisfies \eqref{eq:subharm_bounds}.

(a) If $\Lambda$ is a sampling set for $A^p_{\phi}$ for some $p \in [1, \infty]$,
then it is a
sampling set for all $A^p_\phi$ with $p \in [1, \infty]$.

(b) If $\Lambda$ is an interpolating set $A^p_{\phi}$ for some $p \in [1,
\infty]$, then it is
an
interpolating set for all $A^p_\phi$ with $p \in [1, \infty]$.
\end{theorem}
The proof, which is postponed to Section \ref{sec_loc},  follows from the decay of the reproducing kernel and a
non-commutative Wiener's Lemma.

\subsection{Weak convergence of sets}
Let $\Lambda \subseteq \Rdst$ be a set. A sequence
$\{\Lambda_j: j \geq 1\}$ of subsets of $\Rdst$ \emph{converges
  weakly} to $\Lambda$, in short $\Lambda_j \weakconv \Lambda$, if for every
$R>0$ and
$\varepsilon>0$
  there exists $j_0 \in \Nst$ such that for all $j \geq j_0$,
  \begin{align*}
  \Lambda \cap B_R(0) \subseteq \Lambda_j + B_\varepsilon(0) \quad
  \text{ and }  \quad \Lambda_j \cap B_R(0) \subseteq \Lambda +
B_\varepsilon(0).
  \end{align*}

For a relatively separated set $\Lambda \subseteq \Rdst$, we let $W(\Lambda)$
denote the
set of weak limits of the
translated sets $\Lambda+x, x\in \Rdst$,  i.e.,  $\Gamma \in
W(\Lambda)$ if there exists a sequence
$\sett{x_j: j \geq 1} \subseteq \Rdst$ such that $\Lambda+x_j \weakconv \Gamma$.
It is easy to
see
that then $\Gamma$ is always relatively separated.

  \subsection{Lipschitz convergence of sets}
\label{sec_lip}
  Given a set $\Lambda \subseteq \Rdst$, we say that a sequence of sets
  $\{\Lambda_j: j \geq 1\}$ \emph{converges to $\Lambda$ in a Lipschitz
fashion},
  denoted by  $\Lambda_j \lipconv \Lambda$,
  if there is a sequence of maps $\map_j: \Lambda \to \Rdst$ with the following
properties:
  \begin{itemize}
  \item[(a)] $\Lambda_j = \map_j(\Lambda) = \sett{\mapj{\lambda}: \lambda \in
\Lambda}$.
  \item[(b)] $\mapj{\lambda} \longrightarrow \lambda$, as $j \longrightarrow
\infty$,
  for all $\lambda \in \Lambda$.
  \item[(c)] Given $R>0$,
  \begin{align*}
  \sup_{
  \stackrel{\lambda, \lambda' \in \Lambda}{\abs{\lambda-\lambda'} \leq R}}
  \abs{(\mapj{\lambda} - \mapj{\lambda'}) - (\lambda - \lambda')} \rightarrow 0,
  \quad \mbox {as } j \longrightarrow \infty.
  \end{align*}
  \item[(d)]  Given $R>0$, there exist $R'>0$ and $j_0 \in \Nst$
  such that if
  $\abs{\mapj{\lambda} - \mapj{\lambda'}} \leq R$ for \emph{some} $j
  \geq j_0$ and some $\lambda, \lambda' \in \Lambda$,  then
  $\abs{\lambda-\lambda'} \leq R'$.
  \end{itemize}
   We also say that $\sett{\Lambda_j: j \geq 1}$ is a Lipschitz deformation
   of $\Lambda$, with the understanding that a sequence of underlying maps
$\sett{\map_j: j
\geq 1}$ is also given. We think of each sequence of points
  $\sett{\mapj{\lambda}: j \geq 1}$ as a (discrete) path moving towards the
endpoint
$\lambda$.

The main example of Lipschitz convergence is $\Lambda_j = \tau_j
\Lambda$, where $\tau _j: \Rdst \to \Rdst$,
$\tau_j(0)=0$ and their differential maps $D\tau_j$ satisfy
$D\tau_j \longrightarrow I$ in $L^p(\Rdst,\mathbb{R}^{d\times d})$ for $p>d$
\cite[Lemma 6.4]{grorro15}. In particular if $\{A_j : j \geq 1\}$ is a sequence
of matrices
such that $A_j \longrightarrow I$, then $A_j \Lambda \lipconv \Lambda$.
The notion of Lipschitz deformation is a suitable concept of a global
deformation of sets~\cite{grorro15,MR3500423}. In many situations, as in this article,
Lipschitz deformations preserve sampling sets and interpolating sets.

The following lemma from \cite{grorro15} connects Lipschitz
convergence  and weak
convergence of translates.
\begin{lemma}[Lemma 6.8 in \cite{grorro15}]
\label{lem:weaklim_lip}
Let $\Lambda$ be relatively separated, $\{\Lambda_j:j \geq 1\}$ a Lipschitz
deformation of $\Lambda$, and $\Gamma \subseteq
\Rdst$.
\begin{enumerate}
\item[(i)] For $j \geq 1$ let $\lambda_j \in \Lambda_j$. If
$\Lambda_j - \lambda_j \weakconv \Gamma$, then $\Gamma \in W(\Lambda)$.
\item[(ii)] Let $\{x_j: j \geq 1\} \subseteq \Rdst$ and assume that $\Lambda$ is
relatively
dense. If $\Lambda_j - x_j \weakconv \Gamma$, then $\Gamma \in W(\Lambda)$.
\end{enumerate}
\end{lemma}

\section{Translation type operators}
\label{sec_ops}
We assume that $\phi$ satisfies \eqref{eq:subharm_bounds} and \eqref{eq_a1},
and extend the construction of the translation type operators from
\cite{ortega1998beurling} to several complex variables.

\subsection{Translated weights}
\label{sec_tw}
Given $\zeta \in \Cn$, we let $\phi_{\zeta}$ be a solution of the equation
\begin{equation}
\label{eq_xxx}
\partial \bar \partial \phi_\zeta(z) = \partial \bar \partial \phi(z -\zeta),
\end{equation}
given by Lemma \ref{lem:poincare-lelong} with $\theta = \partial \bar \partial \phi( \cdot - \zeta)$.
Thus, the functions
$\phi_{\zeta}$ satisfy the estimate
\begin{equation} \label{eq:growth_bound}
\phi_{\zeta}(z) \leq C M (1+|z|)^2 \log (1+ |z|)\, .
\end{equation}
We emphasize that the  constant $C$ is independent of
$\zeta$ and  depends only on the dimension $n$.
For each $\zeta \in \Cn$ we thus fix a choice of $\phi_\zeta$ and call it
\emph{translated weight}. For $\zeta=0$, we simply let
\begin{align*}
\phi_0=\phi.
\end{align*}
This choice is possible because we assumed \eqref{eq_a1}.

\subsection{Translation operators}
Let  $q(z, \zeta)$ be a  function that is  entire in $z$ and satisfies
$$
\mathrm{Re}(q(z, \zeta)) : = \phi_\zeta(z) - \phi(z- \zeta).
$$
We now define the translation type operators $T_\zeta$ as
$$
T_\zeta f(z) := e^{q(z,\zeta)} f(z-\zeta).
$$
They satisfy
$$
T_{\zeta}f(z) e^{-\phi_\zeta(z)} = e^{q(z, \zeta)}
f(z-\zeta)e^{-\phi(z-\zeta)-\mathrm{Re}(q(z,\zeta))}= e^{i \mathrm{Im}(q(z, \zeta))}
f(z-\zeta) e^{-\phi(z-\zeta)}.
$$
Consequently,
\begin{align} \label{eq:c3}
\abs{T_{\zeta}f(z)} e^{-\phi_\zeta(z)} =
\abs{f(z-\zeta)} e^{-\phi(z-\zeta)}.
\end{align}
Therefore  $T_\zeta: A^p_\phi \to A^p_{\phi_\zeta}$ is an isometric isomorphism
for all $1 \leq p \leq \infty$. Furthermore,
if $\Lambda$ is a sampling or interpolating set for $A^p_\phi$, then
$\Lambda+\zeta$ is  a sampling or interpolating set
for $A^p_{\phi_\zeta}$ with the same stability constants.
In addition,
\begin{equation}
\label{eq:trans_repro}
\begin{aligned}
K_{\phi_{\zeta}}(z, w) e^{-\phi_{\zeta}(z)-\phi_{\zeta}(w)} &= (T_{\zeta} \otimes
\overline{T_\zeta}) K_\phi (z,w) e^{-\phi_{\zeta}(z)-\phi_\zeta(w)} \\
&= e^{i \mathrm{Im}(q(z,\zeta)- q(w,\zeta))} K_\phi (z-\zeta,
w-\zeta)e^{-\phi(z-\zeta)-\phi(w-\zeta)}.
\end{aligned}
\end{equation}
As a consequence, we have the following covariance formula:
\begin{align}
\label{eq_cov}
T_{\zeta} \left(
e^{-\phi(\lambda)} K_\phi(\cdot,\lambda)
\right)
=
e^{i \mathrm{Im}(q(\lambda+\zeta,\zeta))}
e^{-\phi_{\zeta}(\lambda+\zeta)} K_{\phi_\zeta}(\cdot,\lambda+\zeta).
\end{align}

\subsection{Compactness}
Given a sequence of numbers $\{\zeta_j: j \geq 1\} \subseteq \Cn$, the
family $\{\phi_{\zeta_j}: j \geq 1\}$ satisfies the condition \eqref{eq:subharm_bounds} with
the same constants as $\phi$. As a consequence we prove  the following
compactness result that asserts a  continuous dependence of  the reproducing
kernel $K_\psi $ on $\psi $.

\begin{prop} \label{prop:trans_lim}
Assume that $\phi$ satisfies \eqref{eq:subharm_bounds} and \eqref{eq_a1}.
Then for every  sequence $\{\zeta_j : j \geq 1 \} \subseteq \Cn$ there exists a subsequence
$\{\zeta_{j_k} : k \geq 1 \}$ such that $\phi_{\zeta_{j_k}} $
converges to a plurisubharmonic function $\psi$  uniformly on compact
sets. The function $\psi$ satisfies
\begin{equation} \label{eq:psi_est}
m i \partial \bar \partial |z|^2 \leq i \partial \bar \partial \psi \leq Mi \partial \bar
\partial |z|^2 \, ,
\end{equation}
in the sense of positive currents, and the growth bound \eqref{eq_a1}.
In addition, convergence
$$
K_{\phi_{\zeta_{j_k}}}  \to K_{\psi}
$$
holds uniformly on compact subsets of $\Cn \times \Cn$.
\end{prop}
\begin{proof}
{\em Step 1. (Existence of the convergent subsequence)}. By assumption
and \eqref{eq:growth_bound},
both $\{\phi_{{\zeta_j}}: j \geq 1\}$ and $\{\Delta \phi_{\zeta_j}: j
\geq 1\}$ are locally bounded sequences.
By the regularity of Poisson's equation, the functions $\phi_{\zeta_j}$ belong
locally to $C^{1,\alpha}$ for every  $\alpha \in (0,1)$, i.e.,
\begin{align*}
\sup_{j \geq 1} \sup_{z \in \bC} \norm{\phi_{\zeta_j}}_{C^{1,\alpha}(B_1(z))} < \infty.
\end{align*}
By
the Arzela-Ascoli Theorem and the diagonal argument, it follows that
$\{\phi_{{\zeta_j}}: j \geq 1\}$ has a subsequence that converges locally in the
$C^1$-norm to a certain function $\psi$. In particular,
$\phi_{\zeta_j} \longrightarrow \psi$, $\partial_{z} \phi_{\zeta_j} \longrightarrow
\partial_{z} \psi$,
and $\bar \partial_{z} \phi_{\zeta_j} \longrightarrow \bar \partial_{z} \psi$, uniformly on
compact sets.
With this information, we can deduce \eqref{eq:psi_est} by integrating against a test
function.

{\em Step 2. (Convergence of reproducing kernels)}. By \eqref{eq:growth_bound} and
\eqref{eq:offdiag-bound}, we know that $K_{\phi_{\zeta_j}}$
is locally uniformly bounded in $\Cn \times \Cn$ in the sense that
$\sup _{j \in \Nst } \sup _{(z,w) \in C} |K_{\phi_{\zeta_j}}(z,w)| < \infty $
for every compact set $C\subseteq \Cn \times \Cn$. By Montel's theorem, we can
pass to a subsequence and assume that
$$
K_{\phi_{\zeta_j}} \to K
$$
with uniform convergence on compact sets and a  kernel $K(z,w), z,w\in \Cn$ that is
analytic in $z$ and $\bar{w}$.
We have  to
show that   $K= K_\psi$, and for this  it is enough to show that
$K(z,z)= K_{\psi}(z,z)$, because an entire function in $(z, \bar w)$
is  determined by its  values on the diagonal.

We first  prove $K(z,z) \geq K_{\psi}(z,z)$. We fix $z \in \Cn$ and define $f(w):=
K_{\psi,
z}(w)$. So
$$
\Vert f \Vert_{\psi, 2}^2= f(z)= K_{\psi}(z,z).
$$
For  $\varepsilon > 0$  we choose $R> 0$ such that
$$
\int_{|w-z| > R-1} |f(w)|^2 e^{-2\psi(w)} dm(w) \leq \varepsilon.
$$
We also choose $j_0=j_0(z)$ such that
$e^{-2\phi_{\zeta_j}(w)} \leq 2 e^{-2\psi(w)}$  for all  $j \geq j_0$ and
$w\in B_r(z)$, and
\begin{align}
\label{eq_bounds_ha}
\int_{|w-z| < R} |f(w)|^2 e^{-2 \phi_{\zeta_j}(w)} dm(w) \geq f(z)- 2\varepsilon =
K_\psi(z,z)-2\varepsilon.
\end{align}
Let  $\chi$  be a cut-off  function which
equals $1$ on $B_{R-1}(z)$, $0$ on $\Cn \setminus B_r(z)$ and
$|\bar \partial \chi| \lesssim 1$ everywhere, and set  $h= f\chi$.
Note that, by \eqref{eq_bounds_ha},
\begin{align}
\label{eq_bounds_h}
\norm{h}^2_{\phi_{\zeta_j},2} \leq K_\psi(z,z) + C \varepsilon,
\end{align}
for some constant $C$. We will modify $h$ to a holomorphic function  using
H\"{o}rmander's estimate for $\bar \partial$. This  guarantees   a
solution $u_j$ in $L^2(\Cn , e^{-2\phi _{\zeta _j}})$ of  the
equation $\bar \partial u = \bar\partial h = f \bar \partial \chi$ such that
$$
\int_{\Cn} |u_j(z)|^2 e^{-2\phi_{\zeta_j}(z)} dm(z) \leq C \int_{\Cn} |f(w)|^2 |\bar \partial
\chi(w)|^2
e^{-2 \phi_{\zeta_j}(w)} dm(w) \leq C' \varepsilon \, ,
$$
where the constants $C,C'$ depend on $m$ in \eqref{eq:subharm_bounds}, but not on $j$.
By the choice of $\chi$, $u$ is holomorphic on $B_{R-1}(z)$, therefore, by Lemma \ref{lem:point_eval},
$|u_j(z)|^2 e^{-2\phi _{\zeta _j}(z)} \leq C'' \varepsilon$.
Combining this with \eqref{eq:growth_bound} we conclude that
$$
|u_j(z)|^2 \leq C^2_z \varepsilon,
$$
where $C_z$ depends on $z$.
Since $h(z) = f(z)$, the difference $h_j^*:= h - u_j$ satisfies
$$|h^*_j(z)- f(z)| = |u_j(z) | \leq C_z \varepsilon^{1/2}.
$$
Furthermore, by \eqref{eq_bounds_h},
$$
\Vert h^*_j \Vert_{\phi_{\zeta_j},2} \leq \sqrt{K_{\psi}(z,z)} + C_1
\varepsilon, \quad |h^*_j(z)| = |f(z) -u_j(z)|
\geq
K_\psi(z,z) - C_z \varepsilon ^{1/2},
$$
for some constant $C_1$ which depends only on the constants $m$ and $M$ in
\eqref{eq:subharm_bounds}, and the growth bound \eqref{eq:growth_bound}.
By the extremal characterization of the diagonal values of reproducing kernels, we obtain
\begin{align*}
K_{\phi_{\zeta_j}}(z,z)
= \sup_{g \in A^2_{\phi_{\zeta_j}}} \frac{ |g(z)|^2} {\Vert g \Vert^2_{\phi_{\zeta_j},2}} \geq
\frac{|h_j^*(z)|^2}{\|h_j^*\|_{\phi _{\zeta _j},2}^2} \geq K_\psi(z,z)- C'_z \varepsilon^{1/2},
\end{align*}
where $C'_z$ may depend on $z$.
Since this inequality holds for arbitrarily small $\varepsilon$ and large enough $j$, we deduce
that $K(z,z) \geq K_\psi(z,z)$.

The opposite inequality is obtained similarly by reversing the roles of
$\phi_{\zeta_j}$ and $\psi$.
\end{proof}

In Proposition \ref{prop:trans_lim}, the convergence of the
reproducing kernels holds uniformly on compact sets.
In the next proposition, we show that in certain situations, the convergence of the diagonal entries is in
fact uniform.

\begin{prop}
\label{prop_unif_diag}
Assume that $\phi$ satisfies \eqref{eq:subharm_bounds} and \eqref{eq_a1}. Then, as $\delta \longrightarrow 0$,
\begin{align}\label{eq_aaa}
&K_{(1+\delta)\phi}(z,z)e^{-2(1+\delta)\phi(z)} \longrightarrow K_{\phi}(z,z)e^{-2\phi(z)},
\mbox{ and}
\\
\label{eq_bbb}
&\frac{K_{(1+\delta)\phi}(z,z)e^{-2(1+\delta)\phi(z)}}{K_{\phi}(z,z)e^{-2\phi(z)}} \longrightarrow 1,
\end{align}
uniformly on $ \Cn$.
(Here, $\delta$ may  be positive or negative.)
\end{prop}
\begin{proof}
Arguing as in the proof of Proposition \ref{prop:trans_lim} with
$(1+\delta _j) \phi $ replacing $\phi _{\zeta _j}$, we can show that for any $\varepsilon>0$, there
exists $\delta_0 \in (0,1/4)$ such that
for all $\delta \in (-\delta_0,\delta_0)$,
\begin{align}
\label{eq_zero_diag}
|K_{(1+\delta)\phi}(0,0)e^{-2(1+\delta)\phi(0)}- K_{\phi}(0,0)e^{-2\phi(0)} | < \varepsilon.
\end{align}
The constant $\delta_0$ depends on $\phi$ only through the bounds $m$
and $M$ in \eqref{eq:subharm_bounds}
and the growth bound \eqref{eq:growth_bound}.
Therefore, \eqref{eq_zero_diag} holds also for all weights $(a\phi)_\xi$,
$\xi \in \Cn$, and $a \in (1/2,3/2)$, with the same constant $\delta_0$. By \eqref{eq:trans_repro},
$$
K_{\phi_{-\xi}}(z,z)e^{-2\phi_{-\xi}(z)} = K_{\phi}(z+\xi, z+\xi) e^{-2\phi(z+\xi)}.
$$
As a consequence,
$$
|K_{(1+\delta)\phi}(\xi,\xi)e^{-2(1+\delta)\phi(\xi)}- K_{\phi}(\xi,\xi)e^{-2\phi(\xi)} | < \varepsilon,
$$
for all $\xi \in \Cn$, and all $\delta \in (-\delta_0,\delta_0)$.
Hence, \eqref{eq_aaa} holds uniformly over $\Cn$. This, together with the lower bound in \eqref{eq_diag},
implies that \eqref{eq_bbb} also holds uniformly for $z \in \Cn$.
\end{proof}
\begin{rem} \rm
The uniform convergence
in Proposition \ref{prop_unif_diag} relies on the fact that $\partial \bar{\partial} (1+\delta) \phi$ is
uniformly bounded for small $\delta$.
\end{rem}

\section{Sampling and interpolation}
\label{sec_char}
Whereas Beurling's theory
of sampling and interpolation in Paley-Wiener spaces requires only
weak limits of sets, the theory of weighted Fock spaces requires weak
limits of sets and weight functions. This is the price for the lack of
translation invariance of the $A^2_\phi $'s. Precisely,
given a set $\Lambda$ and a weight $\phi: \Cn \to \Rst$,
we say that $(\Gamma, \psi) \in W(\Lambda, \phi)$ if there exists a sequence $\{ \zeta_j: j \geq
1 \} \subseteq \Cn$ such that $\Lambda + {\zeta_j} \weakconv \Gamma$ and $\phi_{{\zeta_j}} \to \psi$
uniformly on compact sets where  the $\phi_{{\zeta_j}}$'s are the translated weights introduced in
Section \ref{sec_tw}. In what follows, we invoke Theorem
\ref{thm_wiener_samp}  several times. This is
applicable to $\phi$, to the translated weights $\phi_{{\zeta_j}}$, and to its locally uniform limits $\psi$,
because, by Proposition \ref{prop:trans_lim}, they all satisfy
bounds similar to \eqref{eq:subharm_bounds}.

\subsection{Stability of sampling and interpolation under weak limits}

\begin{prop} \label{prop:sampling_char}
Assume that $\phi$ satisfies \eqref{eq:subharm_bounds} and
\eqref{eq_a1}. Let $p \in [1,\infty]$,
$\Lambda$ be a sampling set for $A^p_\phi$,  and suppose that
$(\Gamma, \psi) \in W(\Lambda, \phi)$. Then $\Gamma$ is a sampling set for
$A^p_{\psi}$.
\end{prop}

\begin{proof}
By Theorem \ref{thm_wiener_samp}, we can restrict the problem to $L^2$ norms.
We use of H\"{o}rmander's
$\bar \partial$ estimates, and proceed as in \cite{ortega1998beurling}.

We argue by contradiction and assume  that $\Gamma$ is not a sampling
set for $A^2_\psi$. Let
$\{ \zeta_j: j \geq 1 \} \subseteq \Cn$ be
a sequence such that $\Lambda + \zeta_j \weakconv \Gamma$ and
 $\phi_{\zeta_j} \to \psi$ uniformly on compact sets.
Then for fixed
$\varepsilon
\in (0,1/2)$, we can find $f \in A^2_{\psi}$ such that $\Vert f \Vert_{\psi,2}=1$
and $ \Vert f_{\mid \Lambda} \Vert_{\psi, 2}^2 \leq \varepsilon.$
Take $R > 0$ so large that
$$
\int_{|z| \geq R-3} |f(z)|^2 e^{-2\psi(z)}dm(z) \leq \varepsilon.
$$
We also take a smooth and positive cut-off function $\chi$ such that $\chi=1$ on $B_{R-1}(0)$,
$\chi=0$ on $\Cn \setminus B_R(0)$ and $|\bar \partial \chi| \lesssim 1$. We define
$h= \chi f$. Let $j \geq 1$ be such that
$\Vert h_{\mid \Lambda+\zeta_j}\Vert_{\phi_{\zeta_j},2}^2 \leq 2\varepsilon$, $\Vert h
\Vert_{\phi_{\zeta_j},2 }^2 \geq 1/2$ and
$e^{-2\phi_{\zeta_j}(z)} \leq 2 e^{-2\psi(z)}
$ on $B_R(0)$.

We will  produce an analytic function having properties comparable to those of $h$, thus
giving the
contradiction that we seek. By H\"{o}rmander's estimate
for the $\bar \partial$-operator we can find $u_j \in L^2(e^{-2\phi_{\zeta_j}})$ solving $\bar
\partial u_j = \bar \partial h= f \bar \partial \chi$ such that
\begin{align*}
\|u_j\|_{\phi _{\zeta _j},2}^2 & = \int_{\Cn} |u_j(z)|^2
                                 e^{-2\phi_{\zeta_j}(z)} dm(z) \\
      &\leq \frac{1}{(2m)^n} \int_{\Cn} |f(z)|^2
|\bar \partial \chi (z)|^2 e^{-2\phi_{\zeta_j}(z)} dm(z)
\\ &\lesssim \int_{|z| \geq R-1} |f(z)|^2 e^{-2\psi(z)}dm(z)
\lesssim \varepsilon,
\end{align*}
where the constant $m$ is from \eqref{eq:subharm_bounds}.
Let us consider the sets
\begin{align*}
\Lambda_{j}^1 &:= (\Lambda+ \zeta_j) \cap (B_{R+1}(0) \setminus
                B_{R-2}(0))\, , \\
  \Lambda_j^2 &:= (\Lambda +\zeta_j) \setminus \Lambda_j^1 \, ,
\end{align*}
and the holomorphic function $h_j^*:= h-u_j$, which satisfies
\begin{align}
\label{eq_lpl}
\norm{h_j^*}^2_{\phi_{\zeta_j},2} \geq 1/2 -C \varepsilon,
\end{align}
for some constant $C$.

Note that $u_j$ is holomorphic outside $B_R(0) \setminus B_{R-1}(0)$,
and \[\left(\Lambda_j^2 + B_1(0)\right) \cap \left(B_R(0) \setminus B_{R-1}(0)\right) =
\emptyset,\]
while $\rel(\Lambda_j^2) \leq \rel(\Lambda+\zeta_j) = \rel(\Lambda)$.
Thus, by Corollary \ref{coro_besse},
\begin{align*}
\Vert u_{j \mid \Lambda_j^2} \Vert_{\phi_{\zeta_j},2}^2 \lesssim
 \int_{\Cn} |u_j(z)|^2 e^{-2\phi_{\zeta_j}(z)} dm(z) \lesssim \varepsilon.
\end{align*}
Hence, $\Vert h^*_{j \mid \Lambda_j^2} \Vert_{\phi_{\zeta_j},2}^2 \lesssim \varepsilon$.

To estimate the values of $h^*_j$ on the set $\Lambda^1_j$,
we note that $\left(\Lambda^1_j + B_1(0) \right) \cap B_{R-3}(0)=\emptyset$,
$\rel(\Lambda^1_j) \leq \rel(\Lambda+\zeta_j) = \rel(\Lambda)$, and
$$
\int_{\Cn \setminus B_{R-3}(0)}  |h_j^*(z)|^2 e^{-2 \phi_{\zeta_j}(z)} dm(z) \lesssim
\norm{u_j}_{\phi_{\zeta_j},2}^2+
\int_{\Cn \setminus B_{R-3}(0)}  |f(z)|^2 e^{-2 \psi(z)} dm(z) \lesssim
\varepsilon.
$$
Hence, Corollary \ref{coro_besse} implies that
$\Vert h^*_{j \mid \Lambda^1_j} \Vert_{\phi_{\zeta_j}, 2}^2 \lesssim \varepsilon$.

Since $\Lambda$ is a sampling set for $A^2_\phi$,
the sets $\Lambda+\zeta_j$ are sampling sets for
$A^2_{\phi_{\zeta_j}}$ with the same stability constants for all $j$. However,
we have produced an analytic function $h_j^*$ such that
$\Vert h^*_{j \mid \Lambda+ \zeta_j} \Vert_{\phi_{\zeta_j}, 2}^2 \lesssim
\varepsilon$ and \eqref{eq_lpl}, where the constants are
independent of $\varepsilon$. This contradiction concludes the proof.
\end{proof}
\begin{prop} \label{prop:weaklimit_interp}
Assume that $\phi$ satisfies \eqref{eq:subharm_bounds} and \eqref{eq_a1}.
Let $p \in [1,\infty]$ and $\Lambda$ be an interpolating set for
$A^p_\phi$,
and suppose that $(\Gamma, \psi) \in W(\Lambda, \phi)$. Then $\Gamma$ is an interpolating
set for $A^p_{\psi}$.
\end{prop}
\begin{proof}
We proceed as in \cite{grorro15}.
By Theorem \ref{thm_wiener_samp}, we can restrict the problem to $L^1$-norms.
Let $\{ \zeta_j: j \geq 1 \} \subseteq \Cn$ be
a sequence such that $\Lambda+ \zeta_j \weakconv \Gamma$ and
that $\phi_{\zeta_j} \to \psi$ uniformly on compact sets.

We first  show that for  given $\gamma_0
\in \Gamma$  the interpolation problem
\begin{align} \label{eq:delta_interp}
f(\gamma_0)&= e^{\psi(\gamma_0)},  \\
 f(\gamma)&=0, \quad \gamma \in \Gamma, \gamma \neq \gamma_0 \nonumber
\end{align}
has a solution in $A^1_{\psi}$.

Let $\gamma_j \in \Lambda+ \zeta_j$ be such that $\gamma_j \to \gamma_0$.
Because $\Lambda +\zeta_j$ is an interpolating set for $A^1_{\phi_{\zeta_j}}$ with the same
stability
constant as $\Lambda$ has for $A^1_{\phi}$, we can find functions
$f_j\in A^1_{\phi _{\zeta _j}}$   such that
\begin{align}
  f_j(\gamma_j)& =e^{\phi_{\zeta_j}(\gamma_j)}, \\
  f_j(\gamma)&=0, \qquad \gamma \in \Lambda + \zeta _j, \gamma \neq
               \gamma_j\, , \\
 \Vert f_j \Vert_{\phi_{\zeta_j},1} &\leq
  C  \, ,  \label{eq:interp_uniformbound}
\end{align}
where $C$ is the stability constant of interpolation related to $\Lambda$ in $A^1_{\phi}$.
This, together with Lemma \ref{lem:point_eval} and Montel's theorem, implies the
existence of  a
subsequence of $\{f_j: j \geq 1 \}$ that converges to a holomorphic
function $f = f_{\gamma _0}$
uniformly on compact sets. It is
readily verified that $f \in A^1_{\psi}$ and $\Vert f \Vert_{\psi, 1} \leq C$.
Since $\gamma_j \to \gamma_0$ and $\phi _{\zeta _j} \to \psi $, we
obtain $f(\gamma_0)= e^{\psi(\gamma_0)}$. Second,
since $\Lambda+\zeta_j \weakconv \Gamma$, given $\gamma' \in \Gamma \setminus \{\gamma_0\}$,
there exist $\gamma'_j \in \Lambda+ \zeta_j$ such that $\gamma'_j \to \gamma'$.
For $j \gg 1$, $\gamma'_j \not= \gamma_j$ and, therefore,
$f(\gamma')=\lim_j f_j(\gamma'_j)=0$. Hence, $f_{\gamma _0} $ solves the
interpolation problem \eqref{eq:delta_interp}.

The general interpolation problem is now easily solved. Given a sequence $a \in
\ell^1_{\psi}(\Gamma)$, the series $f= \sum_{\gamma \in \Gamma}  a_\gamma e^{-\psi(\gamma)} f_\gamma$ converges in
$A^1_{\psi}$ by
\eqref{eq:interp_uniformbound} and
therefore uniformly on compact sets. This implies that $f(\gamma_j)= a_j$ for all $j$, as
desired.
\end{proof}

\subsection{Characterization of sampling sets}

\begin{theorem}
\label{th_ch_samp}
Assume that $\phi$ satisfies \eqref{eq:subharm_bounds} and \eqref{eq_a1}.
Then a set $\Lambda$ is a sampling set for $A^p_{\phi}$ if and only if
every pair $(\Gamma, \psi)
\in W(\Lambda, \phi)$ has the property that
$\Gamma$ is a uniqueness set for $A^\infty_\psi$.
\end{theorem}
\begin{proof}
One implication is settled by Theorem \ref{thm_wiener_samp} and
Proposition \ref{prop:sampling_char}.

For the converse, suppose that $\Lambda$ is not a sampling set for
$A^p_\phi$.   We will show that there exists $(\Gamma, \psi) \in W(\Lambda, \phi)$, such that
$\Gamma$ is not a uniqueness set for $A^{\infty}_\psi$.

By Theorem \ref{thm_wiener_samp},  $\Lambda$ is not a sampling set for
$A^{\infty}_{\phi}$. This means that, for every  $j \in \Nst$,
there exists  $f_j \in A^{\infty}_\phi$ such that
$\Vert f_j \Vert_{\phi, \infty}=1$ and $\sup_{\lambda \in \Lambda} |f_j(\lambda)|
e^{-\phi(\lambda)} \leq 1/j$. We select a sequence $\{{\zeta_j}: j\geq 1\} \subseteq \Cn$ such that
$$
[f_j({\zeta_j})| e^{-\phi({\zeta_j})}= |T_{-{\zeta_j}} f_j(0)|
e^{-\phi_{-{\zeta_j}}(0)} \geq 1/2 \, ,
$$
where we have used property \eqref{eq:c3} for the translation
operator. We also have
$$
\sup_{\tau \in \Lambda- {\zeta_j} } |T_{-{\zeta_j}} f_j(\tau)| e^{-\phi_{-{\zeta_j}}(\tau)} \leq 1/j.
$$
By Montel's theorem, the growth bound \eqref{eq:growth_bound} and Proposition
\ref{prop:trans_lim}, we can pass to a subsequence and assume that the following hold: (i)
$\Lambda-{\zeta_j} \weakconv \Gamma$, (ii)  $\phi_{-{\zeta_j}} \to \psi$ uniformly on compact sets
for some plurisubharmonic $\psi$ satisfying \eqref{eq:subharm_bounds}, and (iii)
$T_{-{\zeta_j}} f_j \to f$ uniformly on compact sets for some holomorphic function $f$.

Clearly $f \in A^{\infty}_{\psi}$, $f(0) \neq 0$ and $f_{\mid \Gamma} = 0$. This shows that
$\Gamma$ is not a
uniqueness set for $A^{\infty}_\psi$.
\end{proof}

\subsection{Interpolation and uniqueness}
Interpolating sets can also be characterized with weak limits.
For our purposes, we will need the following technical variation of
\cite[Lemma 5.6]{grorro15}.

\begin{lemma} \label{lem:continuity}
Assume that $\phi$ satisfies \eqref{eq:subharm_bounds} and \eqref{eq_a1}.
Let $\{\Lambda_j: j \geq 1\}$ be a family of separated sets with a uniform separation constant,
i.e.
$$
\inf_j \sep(\Lambda_j)=
\inf\{
\abs{\lambda-\lambda'} :\,
\lambda,\lambda' \in \Lambda_j, \lambda \not= \lambda',
j \geq 1\}
> 0.
$$
Let $c^j \in \ell^{\infty}(\Lambda_j)$, $j \geq 1$, be sequences with $\Vert c^j
\Vert_{\infty}=1$ such that
\begin{equation} \label{eq:zerolimit}
\Vert \sum_{\lambda \in \Lambda_j} c^j_\lambda e^{-\phi(\lambda)} K_{\phi, \lambda}
\Vert_{\phi, \infty} \to 0
\end{equation}
as $j \to \infty$. Then there exists a subsequence
$\{j_k: k \geq 1\} \subseteq \mathbb{N}$; points
$\lambda_{j_k} \in \Lambda_j$; a separated set $\Gamma \subseteq \Cn$;
a nonzero sequence $c \in \ell^{\infty}(\Gamma)$;
and a plurisubharmonic function $\psi$
satisfying the bounds \eqref{eq:subharm_bounds} such that
(i) $\phi_{\lambda_{j_k}} \to \psi$ uniformly on compact sets; (ii) $\Lambda_{j_k}-
\lambda_{j_k} \weakconv \Gamma$; and (iii) the following relation holds
\begin{equation}
\label{eq_conc}
\sum_{\gamma \in \Gamma} c_{\gamma}e^{-\psi(\gamma)} K_{\psi, \gamma} = 0.
\end{equation}
\end{lemma}
\begin{proof}
For each $j$, we select $\lambda_j \in \Lambda_j$ such that
$|c_{\lambda_j}^j|
\geq 1/2$.   Using \eqref{eq_cov}, we can
rewrite the condition \eqref{eq:zerolimit} as
\begin{equation}
\label{eq_tends}
\begin{aligned}
& \sup_{w \in \Cn} \bigg| \sum_{\lambda \in \Lambda_j - \lambda_j} c^j_{\lambda+ \lambda_j}
e^{i \mathrm{Im}(q(\lambda, -\lambda_j))} K_{\phi_{-\lambda_j}}(w,\lambda)
e^{-\phi_{-\lambda_j}(\lambda)-\phi_{-\lambda_j}(w)} \bigg| \\
&\qquad= \Vert  \sum_{\lambda \in \Lambda_j - \lambda_j} c^j_{\lambda+ \lambda_j} e^{i
\mathrm{Im}(q(\lambda, -\lambda_j))} K_{\phi_{-\lambda_j}, \lambda}
e^{-\phi_{-\lambda_j}(\lambda)} \Vert_{\phi_{-\lambda_j}, \infty}
\\
&\qquad=\Vert \sum_{\lambda \in \Lambda_j} c^j_\lambda e^{-\phi(\lambda)} K_{\phi, \lambda}
\Vert_{\phi, \infty} \to 0.
\end{aligned}
\end{equation}
Since the sets $\Lambda^j - \lambda_j$ are uniformly separated, by passing to a subsequence,
we may find a separated set $\Gamma$ such that $\Lambda^j - \lambda_j \weakconv \Gamma$ -
see e.g. \cite[Section 4]{grorro15}.
Now define the sequences
$$
d^j :=
\left(c^j_{\lambda+ \lambda_j} e^{i \mathrm{Im}(q(\lambda, -\lambda_j))}
\right)_{\lambda \in \Lambda_j- \lambda_j}
\in
\ell ^{\infty}(\Lambda_j- \lambda_j)
$$
and consider the associated measure
$$
\mu_j := \sum_{\lambda \in \Lambda_j- \lambda_j} d^j_\lambda \delta_{\lambda}.
$$
These measures  satisfy
$\norm{\mu_j}_{W(\mathcal{M},L^\infty)} \lesssim \rel(\Lambda_j- \lambda_j) \norm{d}_\infty
\lesssim 1$. Thus, by passing to a subsequence, there exists a measure $\mu \in
W(\mathcal{M},L^\infty)$ such that
$\mu_j \rightarrow \mu$ in the $\sigma(W(\mathcal{M},L^\infty), W(C_0,L^1))$-topology.
As shown in \cite[Lemma 4.3]{grorro15}, it follows that $\supp(\mu)
\subseteq \Gamma$, so that we
may write
$$
\mu = \sum_{\gamma \in \Gamma} c_{\gamma} \delta_{\gamma}.
$$
In addition, $\norm{c}_\infty \lesssim \norm{\mu}_{W(\mathcal{M},L^\infty)} < \infty$ by \cite[Lemma 4.6]{grorro15}.
By Proposition \ref{prop:trans_lim},
we may  pass to a further subsequence such that  $\phi_{-\lambda_j} \to \psi$ for some
plurisubharmonic $\psi$
satisfying  \eqref{eq:subharm_bounds} and \eqref{eq:growth_bound}.
By construction $0 \in \Gamma$
and $c(0) = \lim _{j \to \infty} d_{0}^j \neq 0$, therefore  $\mu$ is not identically zero.

Let
$$f_{j,w} (z):= K_{\phi_{-\lambda_j}}(z,w)e^{-\phi_{-\lambda_j}(z)-\phi_{-\lambda_j}(w)}
$$
and
$$
f_w(z) := K_{\psi}(z,w)e^{-\psi(z)-\psi(w)}
$$
be the modified reproducing kernels of $A^2_{\phi _{-\lambda _j}}$ and
$A^2_{\psi }$.
The kernels $K_{\phi_{-\lambda_j}}$ and $K_{\psi}$ satisfy the off-diagonal estimate
\eqref{eq:offdiag-bound} with uniform constants. This fact implies
that $f_{j,w}$ and $f_w$ belong to $W(C_0, L^1)(\mathbb{R}^{2n})$.

With this notation, \eqref{eq_conc} can be recast in terms of the
measure $\mu$  as the statement that $\mu (f_w) = 0$ for all $w$. We
now show that this  is indeed the
case.

Let $w \in \Cn $ and write
$$
\mu(f_w) = \mu_j(f_{j,w}) + \mu_j(f_w-f_{j,w}) + (\mu- \mu_j)(f_w).
$$
The first term tends to zero by our assumption \eqref{eq_tends}, and  the third term tends to zero by the
weak
convergence of $\mu_j$ to $\mu$.

For the second term, we use
Proposition \ref{prop:trans_lim} which says  that $f_{j,w}
\to f_w$
uniformly on compacts, possibly after passing to a further
subsequence. In addition, the uniform off-diagonal decay of the
reproducing kernels
\eqref{eq:offdiag-bound} implies that
$$
| f_w(z)- f_{j,w} (z)| \leq C e^{-c|z-w|},
$$
for some constants $c,C>0$ that are independent of $j$. This
localization estimate and  the uniform   convergence on compact sets
imply  that
$f_{j,w} \rightarrow f_w$ in $W(C_0,L^1)$ as $j \rightarrow \infty$.
Since the weakly convergent sequence $\{\mu_j:j\geq 1\}$ is bounded,
we obtain that
\begin{align*}
\abs{\mu_j(f_w-f_{w,j})} \leq \sup_{j'}
  \norm{\mu_{j'}}_{W(\mathcal{M},L^\infty)} \,
\norm{f_w-f_{w,j}}_{W(C_0,L^1)} \rightarrow 0,
\quad \mbox{as } j\rightarrow \infty \, .
\end{align*}
We have proved that $\mu (f_w) = 0$ for all $w\in \Cn$, which is
\eqref{eq_conc}.
\end{proof}
 As in \cite[Theorem 5.4]{grorro15}, this lemma can also  be used to prove a sharp necessary density condition
for interpolating sets.

\section{Stability of sampling and interpolation under Lipschitz deformations}
\label{th_stab}
We now prove the main result on deformation of sampling and of interpolating sets.
\begin{proof}[Proof of Theorem \ref{th_lip_main}]
By Proposition \ref{prop_a1} we may assume, without loss of generality,
that $\phi$ satisfies the growth condition~\eqref{eq_a1}.

{\em Part (a).} If $\Lambda$ is a sampling set for $A^p_\phi$,
 then $\Lambda $ is also a sampling set for  $A^\infty_\phi$ by
 Theorem \ref{thm_wiener_samp}. We will prove that $\Lambda _j$ is a
 sampling set for $A^\infty _\phi$ for $j \geq j_0$. Applying
 Theorem \ref{thm_wiener_samp} once more, $\Lambda _j$
 is then a sampling set for $A^p_\phi $ for all  $j\geq j_0$.

To prove the claim, we argue by contradiction and, by passing to a subsequence,
assume that none of
the $\Lambda _j$ is a sampling set for $A^\infty _\phi $. Then there exists a sequence of functions $f_j \in
A^{\infty}_\phi$ such that $\Vert f_j \Vert_{\infty, \phi} = 1$ and
$\sup_{z \in \Lambda_j} [f_j(z)|e^{-\phi(z)} \leq 1/j$.  Let ${\zeta_j} \in \bC$ be such that
$|f_j({\zeta_j})|e^{-\phi({\zeta_j})} \geq 1/2$. By Proposition \ref{prop:trans_lim},
we pass to a
subsequence, and assume that (i) the
translates $\phi_{-{\zeta_j}}$ converge to a plurisubharmonic function
$\psi$ uniformly on compact sets,
and that (ii) there exists  a separated set $\Gamma$ such that
$\Lambda_j - {\zeta_j} \weakconv \Gamma$.
See, e.g., \cite[Section 4]{grorro15}.  Since $\Lambda$ is relatively
dense by Corollary \ref{cor_rd} and Theorem \ref{thm_wiener_samp},  Lemma \ref{lem:weaklim_lip} is
applicable and implies that  $\Gamma \in W(\Lambda)$.

By construction,  the translates $T_{-{\zeta_j}} f_j(z)$ satisfy the
following:
\begin{align}
  \Vert T_{-{\zeta_j}} f_j \Vert_{\phi_{-{\zeta_j}}, \infty} &= \Vert
                               f_j \Vert_{\phi, \infty}\, , \notag\\
   |T_{-{\zeta_j}}f_j(0)|e^{-\phi_{-{\zeta_j}}(0)} &= |f_j({\zeta_j})|
  e^{-\phi({\zeta_j})}  \geq 1/2 \, , \label{c5}\\
 \sup_{z \in \Lambda_j - {\zeta_j}} |T_{-{\zeta_j}}
  f_j(z)|e^{-\phi_{-{\zeta_j}}(z)} &= \sum_{z\in \Lambda _j}
                                                     |f_j(z)|^{-\phi (z)}     \leq 1/j \, . \label{c6}
\end{align}
It follows from the growth estimate \eqref{eq_a1} that the sequence $T_{-{\zeta_j}} f_j$
is uniformly bounded on compact sets. Therefore,  Montel's theorem
guarantees the existence of a subsequence converging uniformly on
compact sets to a holomorphic
function $f$, which
is not identically zero by \eqref{c5}. Clearly, the function $f$
belongs to $A^{\infty}_{\psi}$, and \eqref{c6} implies
that $f$ vanishes on $\Gamma \in W(\Lambda)$. By Proposition
\ref{th_ch_samp}, $\Lambda $ is not a sampling set, which contradicts
the assumption. Consequently, $\Lambda _j$ must be a sampling set for $A^\infty _\phi $,
for all for sufficiently large $j$.

{\em Part (b).}
If $\Lambda$ is an interpolating set for $A^p_\phi $,
 then $\Lambda $ is also an interpolating set for  $A^1_\phi$ by
 Theorem \ref{thm_wiener_samp}. We will prove that $\Lambda _j$ is an
 interpolating  set for $A^1 _\phi$ for $j \geq j_0$. Applying
 Theorem \ref{thm_wiener_samp} once more, $\Lambda _j$
 is then an interpolating  set for $A^p_\phi $  for $j\geq j_0$.

Again, we argue by contradiction,
and assume, without loss of generality, that none of the  $\Lambda_j$ is an
interpolating set for $A^1_\phi$. This means that, for every $j\in \Nst
$, the operator
\begin{align*}
A^1_\phi \longrightarrow \ell^1(\Lambda_j), \qquad
f \longmapsto f|_{\Lambda_j}
\end{align*}
fails to be surjective. By duality, it follows that the operator
\begin{align*}
\ell^\infty (\Lambda_j) \longrightarrow  A^\infty _\phi, \qquad
c \longmapsto \sum_\lambda c_\lambda e^{-\phi(\lambda)} K_{\phi ,
  \lambda }
\end{align*}
is not bounded below.
Therefore there exist
sequences $c^j \in \ell^{\infty}(\Lambda_j)$ such that $\Vert c^j \Vert_{\infty} =1$ and
$$
\Vert \sum_{\lambda \in \Lambda_j} c^j_{\lambda} e^{-\phi(\lambda)} K_{\phi, \lambda}
\Vert_{\phi, \infty} \to 0.
$$
Note that $\Lambda$ is separated by Corollary \ref{cor_int_set}.
Since $\Lambda_j \lipconv \Lambda$, by \cite[Lemma 6.7]{grorro15},
we can pass to a further subsequence, and assume that
$\{\Lambda_j: j \geq 1\}$ is uniformly separated. Therefore the assumptions
of  Lemma \ref{lem:continuity} are satisfied. Using
Lemma~\ref{lem:continuity} and  passing to
a further subsequence, we find points $\lambda_j \in \Lambda_j$;
a limiting weight $\phi_{\lambda_j} \longrightarrow \psi$;
a separated set $\Gamma$ such that $\Lambda_j
- \lambda_j \weakconv \Gamma$;
and a nonzero sequence $c \in \ell^{\infty}(\Gamma)$
such that
\begin{equation*}
\sum_{\gamma \in \Gamma} c_\gamma e^{-\psi(\gamma)} K_{\psi, \gamma} =0.
\end{equation*}
Thus, for every $f \in A^1_\psi$,
\begin{equation*}
0=\sum_{\gamma \in \Gamma} c_\gamma e^{-\psi(\gamma)} \ip{f}{K_{\psi, \gamma}}
=\sum_{\gamma \in \Gamma} c_\gamma e^{-\psi(\gamma)} f(\gamma).
\end{equation*}
Since $c \not\equiv 0$, it follows that $\Gamma$ is not an interpolating set
for $A^1_\psi$.
Since  $\Lambda$ is separated, Lemma \ref{lem:weaklim_lip} is
applicable to $\Lambda$ and asserts that $\Gamma \in W(\Lambda)$.
By Proposition
\ref{prop:weaklimit_interp}, $\Lambda$ is not an interpolating set, which contradicts
the assumption. Consequently, $\Lambda_j$ must be an interpolating
set for $A^1 _\phi $
for all sufficiently large $j$.
\end{proof}

As a corollary of the stability with respect to Lipschitz deformations
we show the result announced in the introduction:
sampling and interpolating sets cannot attain the critical density.

\subsection{Proof of Theorem \ref{th_intro}}
\label{sec_th_intro}
By Theorem \ref{thm_wiener_samp}, we can restrict our attention
to $p=2$.
Suppose that $\Lambda$ is a sampling set for $A^2_\phi$ with ${D^{-}_\phi}(\Lambda) \leq 1$,
and consider $\Lambda_j := (1+1/j) \Lambda$. Then $\Lambda_j \lipconv \Lambda$, and by Theorem
\ref{th_lip_main}, $\Lambda_j$ is sampling set for $A^2_\phi$  for
sufficiently large $j$.
The bounds on the diagonal of the reproducing kernel
\eqref{eq_diag} and \eqref{eq:offdiag-bound} imply that ${D^{-}_\phi}(\Lambda_j)<1$
- see Lemma \ref{lemma_xx} below - which contradicts Theorem \ref{th_non_strict}. The
statement about interpolation follows similarly.
$\qed$

\begin{lemma}
\label{lemma_xx}
Assume that $\phi$ satisfies \eqref{eq:subharm_bounds}. Let $\Lambda \subseteq \Cn$ and $a>1$. Then
${D^{-}_\phi}(a \Lambda) < {D^{-}_\phi}(\Lambda)$
and ${D^{+}_\phi}(a \Lambda) < {D^{+}_\phi}(\Lambda)$.
\end{lemma}
\begin{proof}
Using the bounds on the diagonal of the reproducing kernel
\eqref{eq_diag} and \eqref{eq:offdiag-bound}, we estimate
\begin{align*}
&\int_{B_{ar}(z) \setminus B_r(z)} K(w,w) e^{-2\phi(w)} dm(w)
\asymp m(B_{ar}(z) \setminus B_r(z)) \\
&\qquad= (a^{2n}-1)m(B_r(z))
\asymp (a^{2n}-1)
\int_{B_r(z)} K(w,w) e^{-2\phi(w)} dm(w).
\end{align*}
Hence,
\[\int_{B_{ar}(z)} K(w,w) e^{-2\phi(w)} dm(w) \geq
(1 + c) \int_{B_r(z)} K(w,w) e^{-2\phi(w)} dm(w),\]
for some $c=c(a)>0$. As a consequence,
\begin{align*}
{D^-_\phi}(a\Lambda) &=  \liminf_{r \to \infty} \inf_{z \in \Cn} \frac{ \# ( a\Lambda \cap B_r(z))} {
\int_{B_r(z)} K(w,w) e^{-2\phi(w)} dm(w)}
\\
&=  \liminf_{r \to \infty} \inf_{z \in \Cn} \frac{ \# ( \Lambda \cap B_{r/a}(z))} {
\int_{B_r(z)} K(w,w) e^{-2\phi(w)} dm(w)}
\\
&=  \liminf_{r \to \infty} \inf_{z \in \Cn} \frac{ \# ( \Lambda \cap B_r(z))} {
\int_{B_{ar}(z)} K(w,w) e^{-2\phi(w)} dm(w)}
\\
&\leq \frac{1}{1+c} {D^{-}_\phi}(\Lambda) < {D^{-}_\phi}(\Lambda).
\end{align*}
The statement about ${D^{+}_\phi}(\Lambda)$ follows similarly.
\end{proof}

\section{Localizable reproducing kernel Hilbert spaces}
\label{sec_loc}

This section treats the problem of sampling and interpolation in reproducing kernel Hilbert spaces with a
localized reproducing kernel. The results extend considerably those of the
theory of localized frames~\cite{gro04, bacahela06, su06, albakr08, bacala11} and may be of independent
interest.

\subsection{Wiener's lemma with localized subspaces}

\begin{theorem}
\label{th_wiener_sub}
Let $\Lambda, \Gamma \subseteq \Rdst$ be relatively separated, and
let $P: \ell^2(\Gamma) \to \ell^2(\Gamma)$ and
$A:\ell^2(\Gamma) \to \ell^2(\Lambda)$ be bounded operators, represented by matrices
$A=(A_{\lambda, \gamma})_{\lambda \in \Lambda, \gamma \in \Gamma}$ and 
$P = (P_{\gamma, \gamma'})_{\gamma,\gamma' \in \Gamma}$. Assume that
$P^2=P$ and that $A$ and $P$ satisfy the localization estimates:
\begin{align}
\label{eq_loc_A}
&\abs{A_{\lambda, \gamma}} \leq \env(\lambda-\gamma),
\qquad \lambda \in \Lambda, \gamma \in \Gamma,
\\
\label{eq_loc_P}
&\abs{P_{\gamma, \gamma'}} \leq \env(\gamma-\gamma'),
\qquad \gamma,\gamma' \in \Gamma,
\end{align}
for some $\env \in W(C_0,L^1)(\Rdst)$.

Let $p \in [1,\infty]$ and assume that $A$ is $p$-bounded below on the range of $P$, i.e.,
there exists $C_p >0$ such that
\begin{align}
\label{eq_loc_APp}
\norm{A Pc}_p \geq C_p \norm{P c}_p, \qquad c \in \ell^p(\Gamma).
\end{align}
Then there exist $C'>0$ such that for all $q \in [1,\infty]$:
\begin{align}
\label{eq_loc_APq}
\norm{A Pc}_q \geq C' \norm{P c}_q, \qquad c \in \ell^q(\Gamma).
\end{align}
Moreover, $C'$ is independent of $q$  and  depends only on $\env$, $C_p$, and upper bounds for the relative
separation of $\Lambda$ and $\Gamma$.
\end{theorem}
\begin{proof}
When $P=I$ (identity) the result is (a slight extension of) Sj\"ostrand's version of Wiener's
lemma \cite{sj95} in the precise formulation of
\cite[Prop.~A.1]{grorro15}. To prove the result for general $P$,
we  consider the operator
\begin{align}
\tilde{A}:   \ell^2(\Gamma) \to \ell^2(\Lambda) \oplus \ell^2(\Gamma),
\qquad \tilde{A} c = \left( (AP)c, (I-P)c \right).
\end{align}
The matrices representing $AP$ and $I-P$ satisfy enveloping conditions similar to
\eqref{eq_loc_A}, \eqref{eq_loc_P}. To ease the notation,
we keep the same envelope and write:
\begin{align}
\label{eq_loc_Aa}
&\abs{(AP)_{\lambda, \gamma}} \leq \env(\lambda-\gamma),
\qquad \lambda \in \Lambda, \gamma \in \Gamma,
\\
\label{eq_loc_Pa}
&\abs{(I-P)_{\gamma, \gamma'}} \leq \env(\gamma-\gamma'),
\qquad \gamma,\gamma' \in \Gamma.
\end{align}
In order to apply Sj\"ostrand's Wiener-type lemma to $\tilde{A}$, we consider the following
augmented sets:
\begin{align*}
&\Lambda^* := \sett{(\lambda,0): \lambda \in \Lambda} \subseteq \Rst^{d+1},
\\
&\Gamma^* := \sett{(\gamma,1): \gamma \in \Gamma} \subseteq \Rst^{d+1},
\\
&\Upsilon^* := \Lambda^* \cup \Gamma^*.
\end{align*}
Then $\Lambda^*, \Gamma^*,\Upsilon^*$ are relatively separated. Under the identifications
$\ell^2(\Gamma^*) \cong  \ell^2(\Gamma)$
and $\ell^2(\Upsilon^*) \cong \ell^2(\Lambda) \oplus \ell^2(\Gamma)$,
the operator $\tilde{A}$ can be identified with the operator $B: \ell^2(\Gamma^*) \to
\ell^2(\Upsilon^*)$ with matrix entries:
\begin{align}
&B_{(\lambda,0), (\gamma,1)} := (AP)_{\lambda,\gamma},\\
&B_{(\gamma',1), (\gamma,1)} := (I-P)_{\gamma',\gamma}.
\end{align}
 Let $\eta \in C^\infty(\Rst)$ be a cut-off function supported on $[-2,2]$ such that $\eta
\equiv 1$ on $[-1,1]$, and  consider the augmented envelope $\tilde\env \in
W(L^\infty,L^1)(\Rst^{d+1})$ defined by
\begin{align*}
\tilde\env(x,t) := \env(x) \eta(t).
\end{align*}
Using \eqref{eq_loc_Aa} and \eqref{eq_loc_Pa} we see that $B$ satisfies the enveloping
condition:
\begin{align*}
&\abs{B_{(\lambda,0), (\gamma,1)}}
= \abs{(AP)_{\lambda,\gamma}}
\leq
\env(\lambda-\gamma)=\env(\lambda-\gamma)\eta(0-1)
=\tilde\env((\lambda,0)-(\gamma,1)),
\\
&\abs{B_{(\gamma',1), (\gamma,1)}} =
\abs{(I-P)_{\gamma',\gamma}} \leq
\env(\gamma'-\gamma)=\env{(\gamma'-\gamma)} \eta(1-1)
= \tilde\env((\gamma',1)-(\gamma,1)).
\end{align*}
Let us assume \eqref{eq_loc_APp} holds for a certain value of $p \in
[1,\infty]$. We now  estimate
for $c \in \ell^p(\Gamma)$
\begin{align*}
\norm{\tilde{A} c}_{\ell^p \oplus \ell^p}
= \norm{A P c}_{\ell^p} + \norm{(I-P)c}_{\ell^p}
\geq C_p \norm{P c}_{\ell^p} + \norm{(I-P)c}_{\ell^p}
\gtrsim \norm{c}_p;
\end{align*}
see Remark \ref{rem_P}.

Hence $\tilde{A}$ is $p$-bounded below, and therefore so is
$B$. By Sj\"ostrand's Wiener-type lemma
\cite{sj95},\cite[Prop.~A.1]{grorro15}, $B$  and therefore also
$\tilde{A}$  are  $q$-bounded below for
all $q \in [1,\infty]$. Finally, for $c \in \ell^q(\Gamma)$,
\begin{align*}
\norm{AP c}_q = \norm{\tilde{A} Pc}_q \geq C' \norm{Pc}_q,
\end{align*}
as desired. The dependence of the constant $C'$ is discussed in
\cite{grorro15}. In particular, $C'$ is independent of $q$.
\end{proof}
\begin{rem}
\label{rem_P}
{\rm The idempotent $P$ in Theorem \ref{th_wiener_sub}  need
  not  be an orthogonal projection.
The operator $P$ acts on all $\ell^p$ due to the enveloping condition
\eqref{eq_loc_P}, and it is still idempotent on $\ell^p$ by
density. In addition, $\norm{c}_p \asymp \norm{Pc}_p + \norm{(I-P) c}_p$.}
\end{rem}

\subsection{Localizable reproducing kernel Hilbert spaces}
\begin{definition}
\label{loc_subspace}
We say that a (closed) subspace $V \subseteq L^2(\Rdst)$
is a \emph{localizable reproducing kernel Hilbert space} (localizable
RKHS),  if there exist:
(a) a relatively separated $\Gamma \subseteq \Rdst$ (nodes);
(b) a function $\env \in W(L^\infty,L^1)(\Rdst)$ (envelope); and
(c) a frame for $V$, $\Func \equiv \{\func_\gamma: \gamma \in \Gamma\}$,
consisting of continuous functions that
satisfy the localization estimate
\begin{align}
  \label{eq_loc1}
  \abs{\func_\gamma(x)} \leq \env(x-\gamma), \qquad \gamma \in \Gamma.
\end{align}
\end{definition}
We now briefly describe some consequences of the definition. These are part of the abstract theory of
localized frames \cite{fogr05,gro04,bacahela06}. The concrete
application of this theory to the present
setting can be found in first sections of \cite{ro11} (where localizable RKHS are called
\emph{spline-type spaces}). See also \cite{su06} for closely related estimates.

The \emph{canonical dual frame}
$\widetilde\func \equiv \{\tilde\func_\gamma: \gamma \in \Gamma\} \subseteq V$ satisfies a
similar localization
estimate
\begin{align}
    \label{eq_loc2}
  \abs{\tilde\func_\gamma(x)} \leq \tilde\env(x-\gamma), \qquad \gamma \in \Gamma,
\end{align}
for some $\tilde\env \in W(L^\infty,L^1)(\Rdst)$. As a consequence of
the localization estimates \eqref{eq_loc1} and \eqref{eq_loc2}  all operators associated to the frame
$\Func$
obey the expected mapping properties. Specifically, the  \emph{coefficient maps}
$f \mapsto Cf := \left( \ip{f}{\func_\gamma} \right)_{\gamma \in
  \Gamma}$ and $f \mapsto \tilde C f := (
  \langle f, \tilde\func_\gamma \rangle )_{\gamma \in \Gamma}$, which are
bounded from $L^2(\Rdst) \to \ell^2(\Gamma)$ by the frame property,
extend to bounded operators $C, \tilde{C}:L^p(\Rdst) \to
\ell^p(\Gamma)$. Likewise the adjoint operators
\begin{align}
\label{eq_adjoints}
c \mapsto C^* c := \sum_{\gamma \in \Gamma} c_\gamma \func_\gamma
  \quad \text{ and } \,\,  c \mapsto \tilde{C}^* c := \sum_{\gamma \in
  \Gamma} c_\gamma \tilde \func_\gamma
\end{align}
can be extended to bounded operators  $C^*, \tilde{C}^*:
\ell^p(\Gamma) \to L^p(\Rdst)$, $1 \leq p \leq \infty$. Here both
series in \eqref{eq_adjoints} converge
unconditionally in $L^p(\Rdst  )$ for $p <\infty$  and in the weak$^*$-topology for $p=\infty$.

As a consequence, the range-space
\begin{align*}
V^p := C^*(\ell^p(\Gamma))=\left\{
\sum_{\gamma \in \Gamma} c_\gamma \func_\gamma:
c \in \ell^p(\Gamma) \right\} \subseteq L^p(\Rdst )
\end{align*}
is a  well-defined closed, complemented  subspace of $L^p(\Rdst)$ for $1\leq p\leq
\infty$. It follows that the dual space of $V^p$ is $V^{p'}$ where
$1\leq p<\infty$ and
$1/p + 1/p'=1$. These spaces are independent of the particular choice of the frame
$\Func$ in the sense that if two such frames satisfy \eqref{eq_loc1} and
expand the same Hilbert space $V$, then they will produce
the same range spaces $V^p$.
Each function $f \in V^p$ admits the two expansions
\begin{align}
f = C^* \tilde{C} f =
\sum_{\gamma \in \Gamma} \ip{f}{\func_\gamma} \tilde\func_\gamma
= \tilde{C}^* C f
= \sum_{\gamma \in \Gamma} \ip{f}{\tilde \func_\gamma} \func_\gamma \, ,
\end{align}
and the continuity of $C^*$ and $\tilde{C}$ implies that  $C: V^p \to \ell^p(\Gamma)$ is bounded below:
\begin{align}
\norm{Cf}_p \asymp \norm{f}_p, \qquad f \in V^p \, .
\end{align}
Furthermore,  $P=C \tilde{C}^*$ is a projection onto its closed range in $\ell
^p(\Lambda )$ and possesses a matrix representation with entries
$$
P_{\gamma', \gamma} =  \langle \tilde{\func_\gamma} ,
\func_{\gamma'} \rangle,
\qquad \gamma, \gamma' \in \Gamma \, .
$$

Finally, we remark that, on the subspace  $V^p$, the $L^p$ and
$W(L^\infty,L^p)$ norms are equivalent. Since  the
frame elements are continuous,  every function in $V^p$ is continuous,
and therefore
\begin{align}
\label{eq_incl}
V^p \subseteq W(C_0,L^p),
\end{align}
and the following sampling estimate holds:
\begin{align}
\label{eq_samp}
\norm{f|\Lambda}_{\ell^p(\Lambda)}
\lesssim \rel(\Lambda) \norm{f}_{p}, \qquad f \in V^p.
\end{align}
In particular, the \emph{sampling operator}
\begin{align}
\label{eq_samp_op}
S: V^p \to \ell^p(\Lambda),
\qquad
Sf := \left( f(\lambda) \right)_{\lambda \in \Lambda},
\end{align}
is bounded for every relatively separated set $\Lambda$.
(See e.g. \cite{su06, ro11} for proofs.)

The following lemma follows easily from the definitions.
\begin{lemma}
\label{lemma_wc}
Let $V \subseteq L^2(\Rdst)$ be a localizable RKHS.
Suppose that $f_n \longrightarrow f$ in $\sigma(V^\infty,V^1)$,
and $x_n \longrightarrow x$, with $f_n$, $f \in V^\infty$ and $x_n$, $x \in \Rdst$.
Then $f_n(x_n) \longrightarrow f(x)$.
\end{lemma}

\subsection{Universality of sampling and interpolating sets}
The universality of sampling and interpolation in weighted Fock spaces
in Theorem~\ref{thm_wiener_samp} is a special case of the following  much more
general statement about localizable RKHSs.
\begin{theorem}
\label{th_loc_ss}
Let $V \subseteq L^2(\Rdst)$ be a localizable RKHS and
let $\Lambda \subseteq \Rdst$ be relatively separated.

(i) If $\Lambda$ is a sampling set for $V^p$ for some $p \in
[1,\infty]$, then $\Lambda$ is a sampling set for
$V^q$ for all $q \in [1,\infty]$.

(ii) If  $\Lambda$ is an interpolating set for $V^p$ for some $p \in
[1,\infty]$, then $\Lambda$ is an interpolating
set for $V^q$ for all $q \in [1,\infty]$.
\end{theorem}
\begin{proof}

Let $\func \equiv \{\func_\gamma: \gamma \in \Gamma\}$ be
a localized frame, as in Definition \ref{loc_subspace}.

  (i) Assume that $\Lambda$ is a sampling set for $V^p$ for some $p\in
  [1,\infty ]$ and  let $q \in [1,\infty]$. We must show that $\Lambda$ is a
sampling set for $V^q$.
In addition to  the orthogonal projection $P=C \tilde{C}^*$
 and the sampling operator $S$ from  \eqref{eq_samp_op},
we  define
$A = S \tilde{C}^*$. Written as a  matrix,  $A$ has the entries
\begin{align*}
     A_{\lambda, \gamma} = \widetilde{\func_\gamma}(\lambda),
\qquad \lambda \in \Lambda, \gamma \in \Gamma \, .
\end{align*}
By \eqref{eq_loc1} and \eqref{eq_loc2}  $A$ and $P$ satisfy the decay
estimates \eqref{eq_loc_A} and \eqref{eq_loc_P}.

In order to apply Theorem \ref{th_wiener_sub}, we will   verify the
stability condition \eqref{eq_loc_APp}. Given $c \in \ell^p(\Gamma)$, let
$f = \tilde{C}^* c \in V^p$. Since $\Lambda$ is a sampling set for
$V^p$, we have
$\norm{Sf}_p \asymp \norm{f}_p$. Using the properties of the operators
$C,\tilde{C}$, we obtain the estimate
\begin{align}
\norm{APc}_p&=\norm{S \tilde{C}^* C \tilde{C}^* c}_p
= \norm{S \tilde{C}^* C f}_p = \norm{S f}_p \notag \\
&\asymp \norm{f}_p \asymp \norm{C f}_p
= \norm{C \tilde{C}^* c}_p = \norm{P c}_p.  \label{c7}
\end{align}
At this point  we can invoke Theorem \ref{th_wiener_sub} and  conclude that
$\norm{APc}_q \gtrsim \norm{Pc}_q$ for all $q\in [1,\infty]$.

To show that $\Lambda$ is a sampling set for $V_q$, let  $f \in V_q$
and  $c := Cf$,
so that $f=\tilde{C}^* c$. We repeat  the estimates \eqref{c7} and
obtain
\begin{align*}
\norm{Sf}_q &= \norm{S (\tilde{C}^* C) f}_q
=\norm{S (\tilde{C}^* C) \tilde{C}^* c}_q
\\
&=\norm{(S \tilde{C}^*) (C \tilde{C}^*) c}_q
=\norm{A P c}_q\\
&\gtrsim \norm{P c}_q = \norm{C \tilde{C}^* C f}_q = \norm{Cf}_q
\asymp \norm{f}_q,
\end{align*}
as desired.

(ii) By definition, $\Lambda$ is an interpolating
set for $V^p$,  if
the sampling operator $S: V^p \to \ell^p(\Lambda)$ is surjective. This is the case if and only if
$S \tilde{C}^*: \ell^p(\Gamma) \to \ell^p(\Lambda)$ is surjective, which in turn holds if and
only if
$\tilde{C} S^*: \ell^{p'}(\Lambda) \to \ell^{p'}(\Gamma)$ is bounded below, where $1/p + 1/p'
=1$.
(For the case $p=\infty$, we use the fact the the operators are weak$^*$-continuous.)
The operator
$\tilde{C} S^*$ is represented by the matrix with entries
\begin{align*}
A^*_{\gamma, \lambda} = \overline{\widetilde{\func_\gamma}(\lambda)},
\qquad \lambda \in \Lambda, \gamma \in \Gamma.
\end{align*}
We invoke again Theorem \ref{th_wiener_sub} --- this time with $P=I$
--- and  conclude that
if $S^* \tilde{C}$ is $p$-bounded below for some $p \in [1,\infty]$, then it is $p$-bounded below
for all $p \in [1,\infty]$. This concludes the proof.
\end{proof}

\subsection{Sets close to being sampling and interpolating}
We say that a localizable RKHS is uniformly localizable if
\begin{align}
\label{eq_uq}
\sup_{\stackrel{x,y \in \Rdst}{\abs{x-y} \leq \delta}}
\sup_{\stackrel{f \in V^\infty}{\norm{f}_\infty \leq 1}}
\big|{\abs{f(x)}-\abs{f(y)}}\big|
\longrightarrow 0,
\mbox{ as } \delta \longrightarrow 0^+.
\end{align}
Note that the uniformity property concerns the \emph{absolute values} of the functions in
$V^\infty$. It is thus a weaker condition than the uniform equicontinuity of the ball of $V^\infty$. For 
comparison,
the equicontinuity of the ball of $V^\infty$ amounts to the condition
\begin{align}
\label{eq_str_unif}
\sup_{\stackrel{x,y \in \Rdst}{\abs{x-y} \leq \delta}} \norm{K_x-K_y}_1 \longrightarrow 0,
\mbox{ as } \delta \longrightarrow 0^+\,
\end{align}
for the reproducing kernel of $V^\infty$.
In several examples, in particular in spaces of analytic functions,
the uniformity property \eqref{eq_uq} is easier to verify than \eqref{eq_str_unif}.
See Proposition \ref{prop_loc_fock} below.

The following is an abstract version of the construction in \cite{Nir}. Sampling sets with close to critical
density can also be obtained from the general result in \cite{bacala11}.
\begin{theorem}
\label{th_fekete}
Let $V \subseteq L^2(\Rdst)$ be a uniformly localizable RKHS.
Then there exists a separated set $\Lambda \subseteq \Rdst$ with the following properties.

\begin{itemize}
\item[(i)] (\emph{$\ell^1-\ell^\infty$ interpolation}.) There exist a collection of functions
$\{l_\lambda: \lambda \in \Lambda\} \subseteq V^\infty$ such that
$\norm{l_\lambda}_\infty = 1$, and
\begin{align}
\label{eq_ll}
l_\lambda(\lambda') = \delta_{\lambda, \lambda'},
\qquad
\lambda, \lambda' \in \Lambda.
\end{align}
In particular, given  $a \in \ell^1(\Lambda)$,
$f = \sum_{\lambda \in \Lambda} a_\lambda l_\lambda \in V^\infty$ and satisfies
$f(\lambda)=a_\lambda$ for all $\lambda \in \Lambda$.

\item[(ii)] (\emph{$\ell^1-\ell^\infty$ sampling}.) For all $f \in V^1$
\begin{align}
\label{eq_int}
\norm{f}_\infty \leq \sum_{\lambda \in \Lambda} \abs{f(\lambda)}.
\end{align}
\end{itemize}
\end{theorem}
\begin{proof}
{\em Step 1 (Construction of Fekete points and  the set $\Lambda $)}.
Without loss of generality we assume that $\dim V = \infty$.
Let $\{P_n: n \geq 1\} \subseteq V^1$ be a nested sequence of subspaces
such that $\dim P_n =n$ and $\bigcup_n P_n$ is dense in $V^1$.
Let $\{p^n_1,\ldots,p^n_n\}$ be an orthogonal basis of $P_n$.
Consider the functional
\[
\Delta(x_1,\ldots, x_{n}) = \abs{\det \left(p^n_i(x_j) \right)_{i,j=1\ldots n}},
\qquad x_j \in \Rdst,
\]
and let $(x^n_1,\ldots,x^n_n)$ be a maximizer of $\Delta$. We denote
the corresponding set of points by  $\Lambda^n = \{x_1^{n},\ldots,
x_{n}^{n}\}\subseteq \mathbb{R}^d $.
A  maximizer of $\Delta $  always  exists because, by \eqref{eq_incl}, every function in $V^1$ vanishes at infinity, and
therefore it is enough to maximize $\Delta$ on a suitable compact set.

We consider the Lagrange functions $l_1^{n}, \cdots, l_{n}^{n}$ defined
as
\[
 l_{x^n_i}(x) =   \frac{
 \begin{vmatrix}
  p^n_1(x_1^{n})&\ldots &p^n_1(x)& \ldots& p^n_1(x_{n}^{n})\\
  \vdots &&  \vdots&& \vdots\\
  p^n_{n}(x_1^{n})&\ldots &p^n_{n}(x)& \ldots& p^n_{n}(x_{n}^{n})\\
 \end{vmatrix}   }{
 \det \left(p^n_i(x_j^{n}) \right)_{i,j=1\ldots n}}.
\]
The functions $\{l^n_\lambda : \lambda \in \Lambda^n\}$ form a basis of $P_n$, and
satisfy $l_\lambda^{n}(\mu) = \delta_{\lambda,\mu}$, for all $\lambda, \mu \in \Lambda^n$.
In addition, since $(x^n_1,\ldots,x^n_n)$  is a maximizer for
$\Delta$, we know that  $\norm{l_\lambda^{n}}_\infty =1 $.

Next, for distinct $\lambda, \mu \in \Lambda^n$,
$1 = \big|\abs{l_\lambda(\lambda)}-\abs{l_\lambda(\mu)}\big|$.
By the uniformity property \eqref{eq_uq}, it follows that the sets
$\Lambda^n$ are uniformly separated, i.e.,
\begin{align*}
\inf _{n\in \mathbb{N}} \inf \{\abs{\lambda-\mu}:  \lambda, \mu \in \Lambda^n, \lambda \not=\mu\} >0.
\end{align*}
By passing to a subsequence, we may assume that $\Lambda^n \weakconv \Lambda$, for some
separated set $\Lambda$, and because of the uniform
separation, the associated measures converge in the following manner:
\begin{align}
\label{eq_mun_mu}
\mu _n :=
\sum_{\lambda \in \Lambda^n} \delta_\lambda \longrightarrow
\mu:= \sum_{\lambda \in \Lambda} \delta_\lambda,
\qquad \mbox{ in } \sigma(\mathcal{M}, W(C_0,L^1)),
\end{align}
see, e.g., \cite[Section 4]{grorro15}.

{\em Step 2 (Construction of the dual system)}. Let $\lambda \in
\Lambda $,
then there exists a sequence $\lambda^n \in \Lambda ^n $ such that
$\lambda^n \longrightarrow \lambda$.
By passing to a subsequence we may assume that $l^n_{\lambda^n} \longrightarrow l_\lambda$ in
$\sigma(V^\infty,V^1)$, for some $l_\lambda \in V^\infty$. Note that
$\norm{l_\lambda}_\infty \leq \liminf_n \norm{l^n_{\lambda^n}}_\infty =1$.
Using Lemma \ref{lemma_wc}, it follows that \eqref{eq_ll} holds. (See also the proof
of Proposition \ref{prop:weaklimit_interp}.)

{\em Step 3 (Interpolation)}. Given $a \in \ell^1(\Lambda)$ we let $f := \sum_{\lambda} a_\lambda
l_\lambda$. Then the series converges absolutely, $\norm{f}_\infty \leq \norm{a}_1$, and
$f(\lambda)=a_\lambda$, for all $\lambda \in \Lambda$.

{\em Step 4 (Sampling)}. Fix $n \geq 1$ and let $f \in P_n$.
For all $N \geq n$, since $f \in P_N$,
$f(x) = \sum_{\lambda \in \Lambda^N} f(\lambda) l^N_\lambda(x)$.
By \eqref{eq_incl}, $f \in V^1 \subseteq W(C_0, L^1)$,
and, by \eqref{eq_mun_mu},
\begin{align*}
\norm{f}_\infty \leq \sum_{\lambda \in \Lambda^N} \abs{f(\lambda)}
= \int \abs{f} d{\mu_N} \longrightarrow
\int \abs{f} d\mu
= \sum_{\lambda \in \Lambda} \abs{f(\lambda)}.
\end{align*}
Since $\bigcup_n P_n$ is dense in $V^1$ and \eqref{eq_samp}
holds, the full sampling estimate \eqref{eq_int} now follows.
\end{proof}

\subsection{Application to Fock spaces}
We now apply the results about localizable RKHSs to weighted Fock
spaces. First we need to show that every weighted Fock space $A_\phi
^2$ is a localizable RKHS.
\begin{prop}
\label{prop_loc_fock}
Let $\phi: \Cn \to \mathbb{R} $ be a plurisubharmonic function satisfying
\eqref{eq:subharm_bounds}
and consider the weighted Fock space $A^2_{\phi}$. Then the space
\begin{align*}
V^2_\phi := \sett{f= g e^{-\phi}: g \in A^2_{\phi}}
\end{align*}
is a uniformly localizable RKHS in $L^2(\mathbb{R}^{2n})$.
\end{prop}
\begin{proof}
Let $d=2n$ and let $P:L^2(\Rdst) \to V^2_\phi$ be the orthogonal
projection. If $K_\phi $ is the kernel of $A^2_\phi$, then the
reproducing kernel of $V^2_\phi$ is given by
\begin{equation}
  \label{eq:ll1}
 \widetilde{K_\phi}(z,w) = K_\phi (z,w)
e^{-\phi (z) - \phi (w)}
\end{equation}
 Therefore the off-diagonal decay estimate
for the producing kernel of $A^2_{\phi}$ in \eqref{eq:offdiag-bound}
reads as
\begin{align}
\label{eq_K}
\abs{\widetilde{K_\phi}(z,w)}\lesssim e^{-c\abs{z-w}}, \qquad z,w \in \Rdst,
\end{align}
for some constant $c>0$.
Let $\delta \in (0,2/\sqrt{d})$, $\Gamma := \delta \Zdst$,
$I:=[-1/2,1/2]^d$ (so that $\delta I \subseteq B_1(0)$ for $\delta <
2/\sqrt{d}$), and
\begin{align*}
\func_\gamma := P(\delta^{-d} 1_{\gamma + \delta I}), \qquad \gamma \in \Gamma.
\end{align*}
By \eqref{eq_K},
\begin{align*}
\abs{\func_\gamma(z)} \lesssim \delta^{-d} \int_{\Rdst} 1_{\gamma+\delta I}(w) e^{-c\abs{z-w}}
dm(w) \leq
C_\delta
e^{-c \abs{z-\gamma}},
\end{align*}
for some constant $C_\delta>0$. Hence the family $\Func \equiv \{\func_\gamma: \gamma \in
\Gamma\}$ satisfies the localization estimate~\eqref{eq_loc2}.

Let us show that for suitably small   $\delta \in (0,2/\sqrt{d})$,
$\Func$ forms a frame of $V^2_\phi$.
Let $f \in V^2_\phi$ and define
\begin{align*}
\tilde{f} := \sum_{\gamma \in \Gamma} \delta^{-d} \left(\int_{\delta I+\gamma} f \right)  1_{\delta I+\gamma}
= \sum_{\gamma \in \Gamma} \ip{f}{\func_\gamma}  1_{\delta I+\gamma}.
\end{align*}
Using the mean value theorem, \eqref{eq:c2} of Lemma
\ref{lem:point_eval} with $p=r=2$, and the fact that $\delta<2/\sqrt{d}$, we
obtain the pointwise estimate
\begin{align*}
\abs{f-\tilde{f}}^2 1_{\delta I+\gamma} \lesssim \delta ^2
 \int_{B_{2}(\gamma)} \abs{f(z)}^2 \, dm(z),
\end{align*}
and consequently
\begin{align}
\norm{f-\tilde{f}}^2_2 \lesssim \delta^2 \delta^{d} \sum_{\gamma \in \Gamma} \int_{B_{2}(\gamma)}
\abs{f}^2
\lesssim \delta^2 \norm{f}^2_2. \label{c8}
\end{align}
Choosing $\delta \ll 1$ small enough,  we conclude that
\begin{align*}
\norm{f}_2 \leq 2 \norm{\tilde{f}}_2 = 2 \delta^{d/2} \left( \sum_{\gamma \in \Gamma}
\abs{\ip{f}{\func_\gamma}}^2
\right)^{1/2}.
\end{align*}
The converse inequality
$\sum_{\gamma \in \Gamma} \abs{\ip{f}{\func_\gamma}}^2= \delta ^d
\|\tilde{f}\|_2^2 \leq \tfrac{3}{2}
\norm{f}^2_2$ follows also
from \eqref{c8}.
Hence  $\Func$ is a frame of $V^2(\Cn)$, as claimed.

Finally, if $f \in V^\infty_\phi$, i.e., $f=f_0 e^{-\phi }$ for $f_0
\in A^\infty _\phi $,  then by Lemma \ref{lem:point_eval}, with $r=p=1$,
\begin{equation*}
|\nabla \left(|f(z)| \right)| \lesssim \bigg[ \int_{B_1(z)}
|f(w)| dm(w) \bigg] \lesssim \norm{f}_\infty.
\end{equation*}
This implies the  property of uniform localization~\eqref{eq_uq}.

\end{proof}

\subsection{Proof of Theorem \ref{thm_wiener_samp}}
By  Proposition \ref{prop_loc_fock}, $V^2_\phi $ is a uniformly
localizable RKHS.   Theorem \ref{th_loc_ss} then  assert
the universality of sampling and interpolation for $V^p_\phi $ for all
$p\in [1,\infty ]$, which  is the same as universality for the weighted Fock spaces
$A^p_\phi $. $\qed$

\subsection{Proof of Theorem  \ref{sharp}}
\mbox{}

{\em Step 1}.
Again we  argue in terms of the reweighted spaces
\begin{align*}
V^p_{\phi} := \{f= g e^{-\phi}: g \in A^p_{\phi} \} \subseteq
  L^p(\mathbb{R}^d) \, .
\end{align*}
According to \eqref{eq_diag} and \eqref{eq:offdiag-bound}, the
reproducing kernel $\widetilde{K_\phi}$ satisfies the estimates
$\widetilde{K_\phi}(z,z) \asymp 1$ and $\widetilde{K_\phi}(z,w)
\lesssim e^{-c\abs{z-w}}$ for some constant $c>0$. These estimates
hold,  of course, also for  $\widetilde{K_{a\phi}}$ for all $a>0$ with possibly different constants.

Fix $\varepsilon >0$. To prove Theorem~\ref{sharp}, we   need to produce a set $\Lambda
\subseteq \Cn$ that  is  interpolating for
$A^p_{(1+\varepsilon)\phi}$ and  a sampling
set for  $A^p_{(1-\varepsilon)\phi}$ with  $1\leq p\leq \infty$.  By Proposition \ref{prop_loc_fock}
and Theorem \ref{th_loc_ss}, it suffices to produce  $\Lambda$ that is an
interpolating set for $V^1_{(1+\varepsilon)\phi}$ and a sampling set for
$V^\infty_{(1-\varepsilon)\phi}$.

We now show that the set $\Lambda $ constructed in  Theorem
\ref{th_fekete} as a weak limit of Fekete points does the job. Theorem
\ref{th_fekete} yields a set $\Lambda$ and interpolating functions
$\{l_\lambda: \lambda \in \Lambda\} \subseteq V^\infty_\phi$ such that
$\|l_\lambda \|_\infty \leq 1$ and $l_\lambda (\mu) = \delta
_{\lambda , \mu}$ for all $\lambda,\mu \in \Lambda $.

(i) \emph{Interpolation.}
We improve the localization of the functions $l_\lambda$  in the following way: let
\begin{align*}
\widetilde{l}_\lambda(z) := l_\lambda(z) \frac{\widetilde{K_{\varepsilon \phi}}(z,\lambda)}
{\widetilde{K_{\varepsilon \phi}}(\lambda,\lambda)}, \qquad z \in \Cn.
\end{align*}
Then clearly $\widetilde{l}_\lambda(\mu)=\delta_{\lambda,\mu}$, for all $\lambda,\mu \in \Lambda$,
and, due to the decay of $\widetilde{K_{\varepsilon \phi}}$,
\begin{align}
\label{eq_nnn}
\abs{\widetilde{l}_\lambda(z)} \lesssim
e^{-c\abs{z-\lambda}}.
\end{align}
Consequently,  we have
$$
\|\widetilde{l_\lambda}\|_1 \lesssim 1 \, .
$$
Furthermore, since  $l_\lambda = h e^{-\phi} $ for some $h\in
A^\infty_\phi $ and  $\widetilde{K_{\varepsilon\phi}}(z,\lambda ) =
e^{-\varepsilon\phi(\lambda )} K_{\varepsilon\phi}(z,\lambda)
e^{-\varepsilon\phi(z)} $ with $K_{\varepsilon\phi,\lambda } \in
A^1_{\varepsilon\phi}$,  the product   $\widetilde{l_\lambda } =
l_\lambda \widetilde{K_{\varepsilon \phi ,\lambda } }$ is in $V^1_{(1+\varepsilon)\phi}$.

Thus, for  $a \in \ell^1(\Lambda)$ the function  $f(z) :=
\sum_{\lambda \in \Lambda} a_\lambda \widetilde{l_\lambda}$ is in $
V^1_{(1+\varepsilon)\phi}$ and satisfies $f(\lambda)=a_\lambda$, as desired.

(ii) \emph{Sampling.}
Here we   want to verify the inequality
$\sup_{z \in \Cn} \abs{f(z)} \leq \sup_{\lambda \in \Lambda}
\abs{f(\lambda)}$ for every  $f \in V^\infty_{(1-\varepsilon)\phi}$. To this end, fix $z_0 \in \Cn$
and define
\begin{align*}
g(z) := f(z) \frac{\widetilde{K_{\varepsilon \phi}}(z,z_0)}
{\widetilde{K_{\varepsilon \phi}}(z_0,z_0)}.
\end{align*}
As in (i) we see that  $g \in V^1_{\phi}$. By
the $\ell^1$-$\ell ^\infty$-sampling part of Theorem~\ref{th_fekete}
applied to $g$  we obtain
\begin{align*}
\abs{f(z_0)}=\abs{g(z_0)}
\leq \sum_{\lambda \in \Lambda} \abs{g(\lambda)}
\lesssim \sum_{\lambda \in \Lambda} \abs{f(\lambda)} e^{-c\abs{z_0-\lambda}}
\leq C_\Lambda \sup_{\lambda \in \Lambda} \abs{f(\lambda)},
\end{align*}
where $C_\Lambda := \sup_{z \in \Cn} e^{-c\abs{z-\lambda}}$ is finite because $\Lambda$ is separated.

{\em Step 2}.
Let us finally show that \eqref{eq_infsup} holds.
We consider only the supremum $\sup_{\Lambda \in SI} D^{-}(\Lambda)=1$
over all interpolating sets for $A^2_\phi$; the argument for the infimum is analogous.
By Theorem \ref{th_non_strict}, the supremum in \eqref{eq_infsup} is at most 1.
To show that it is indeed 1, we
fix $\delta \in (0,1)$ and use
Proposition \ref{prop_unif_diag} to select $\varepsilon>0$ such that
$$\widetilde{K}_{(1-\varepsilon)/(1+\varepsilon)\phi}(z,z) \geq (1-\delta)
\widetilde{K_{\phi}}(z,z),$$ for all $z \in \Cn$. We now apply the construction
from Step 1 to the weight $(1+\varepsilon)^{-1}\phi$ and obtain a separated set $\Lambda \subseteq \Cn$ that is
 interpolating for $A^2_\phi$ and sampling for
$A^2_{(1-\varepsilon)\phi /(1+\varepsilon)}$.
By Theorem \ref{th_non_strict}, we conclude that
$D^{-}_{(1-\varepsilon)\phi /(1+\varepsilon)}(\Lambda) \geq 1$.
Hence, for $R \gg 1$, and all $z\in\Cn$,
\begin{align*}
\# \left(\Lambda \cap B_R(z) \right)
&\geq (1-\delta) \int_{B_R(z)} \widetilde{K}_{(1-\varepsilon)\phi /(1+\varepsilon)}(w,w) dm(w)
\\
&\geq (1-\delta)^2 \int_{B_R(z)} \widetilde{K_{\phi}}(w,w) dm(w).
\end{align*}
This means that
$D^{-}_{\phi}(\Lambda) \geq (1-\delta)^2$. Since $\delta \in (0,1)$
was arbitrary,   the conclusion follows.
$\qed$

\end{document}